\documentclass[12pt,reqno] {amsart}%
\usepackage{amsfonts}
\usepackage{amsmath}
\usepackage{amssymb}
\usepackage{xypic,color}
\usepackage{graphicx}%
 \theoremstyle{plain}

\newtheorem{theorem}{Theorem}[section]

\newtheorem{corollary}[theorem]{Corollary}
\newtheorem{definition}[theorem]{Definition}
\newtheorem{lemma}[theorem]{Lemma}
\newtheorem{proposition}[theorem]{Proposition}
\newtheorem{remark}[theorem]{Remark}
\newtheorem{example}[theorem]{Example}

\def\bn{\begin{definition}}
\def\en{\end{definition}}
\def\ba{\begin{array}}
\def\ea{\end{array}}
\def\be{\begin{equation}}
\def\ee{\end{equation}}
\def\bd{\begin{description}}
\def\ed{\end{description}}
\def\bu{\begin{enumerate}}
\def\eu{\end{enumerate}}
\def\bi{\begin{itemize}}
\def\ei{\end{itemize}}

\newcommand{\A}{\mathcal A}

\newcommand{\B}{\mathcal B}

\def\diag{{\rm diag}}

\def\bN{\mathbb{N}}
\def\cN{\mathcal{N}}

\def\bZ{\mathbb Z}
\def\bQ{\mathbb Q}

\def\cD{\mathcal{D}}
\def\bD{\mathbb D}

\def\cA{\mathcal{A}}

\def\cP{\mathcal{P}}
\def\cF{\mathcal{F}}
\def\cG{\mathcal{G}}

\def\cU{\mathcal{U}}
\def\cV{\mathcal{V}}
\def\cW{\mathcal{W}}

\def\cE{\mathcal{E}}

\def\ds{\displaystyle}

\def\i{\mathfrak{i}}

\def\<{\langle}
\def\>{\rangle}
\usepackage[]{hyperref}
\begin{document}
\title[]{Anqie entropy and arithmetic compactification of natural numbers}

\author[]{Fei Wei}

\address{Yau Mathematical Sciences Center, Tsinghua University, Beijing 100084, China}

\email{weif@mail.tsinghua.edu.cn}

\maketitle
\begin{abstract}
To study arithmetic structures of natural numbers, we introduce a notion of entropy of arithmetic functions,
called anqie entropy. This entropy possesses some crucial properties common to both Shannon's and Kolmogorov's entropies.
We show that all arithmetic functions with zero anqie entropy form a C*-algebra. Its maximal ideal space defines
our arithmetic compactification of natural numbers, which is totally disconnected but not extremely disconnected.
We also compute the $K$-groups of the space of all continuous functions on the arithmetic compactification.
As an application, we show that any topological dynamical system with topological entropy $\lambda$, can be approximated by symbolic dynamical systems with entropy less than or equal to $\lambda$.
\end{abstract}
\section{introduction}
Let $\bN=\{0,1,2,\ldots\}$ be the set of natural numbers. Complex-valued functions defined on $\bN$ are usually called \emph{arithmetic functions}.  They are important in the study of the distribution of primes and other arithmetic problems (e.g., the Prime Number Theorem, the Twin Prime Conjecture, etc). In this paper, we shall apply tools from operator algebra and dynamical systems to study arithmetic functions.
We start from C*-subalgebras of $l^{\infty}(\mathbb{N})$, the algebra of all bounded arithmetic functions, which is an abelian C*-algebra acting on $l^2(\mathbb{N})$.

Suppose that $\mathcal{A}$ is a unital C*-subalgebra of $l^{\infty}(\mathbb{N})$, then there is a compact Hausdorff space $X$ such that $\mathcal{A}$ is *-isomorphic to $C(X)$ by Stone-Gelfand-Naimark theory, where $C(X)$ is the space of all continuous complex-valued functions on $X$. The space $X$ is known as the maximal ideal space of $\mathcal{A}$. Moreover, if $\mathcal{A}$ is generated by a bounded arithmetic function $f$, its maximal ideal
space $X$ is homeomorphic to the closure of $f(\mathbb{N})$ in $\mathbb{C}$. For example, the C*-algebras generated by $\chi_{\mathcal{P}}$, the characteristic function defined on the set of primes, and by $f(n)=(-1)^{n}$ are *-isomorphic, as their maximal spaces are homeomorphic to the space with two points.
However, it is obvious that $\chi_{\mathcal{P}}$ carries more arithmetic information than $(-1)^{n}$. So generally, C*-algebras may not reflect arithmetic properties of $\mathbb{N}$. We are interested in  C*-subalgebras of $l^{\infty}(\mathbb{N})$, which can preserve arithmetic structures of natural numbers, thus we introduce the following concept of anqie.

\begin{definition}\label{definition of anqie}
Let $\mathcal{A}$ be a unital C*-subalgebra of $l^{\infty}(\mathbb{N})$. We call $\mathcal{A}$ \textbf{an anqie} (of $\mathbb{N}$) if $\mathcal{A}$ is invariant under the map $\sigma_{A}$ on $l^{\infty}(\mathbb{N})$, given by $(\sigma_{A}f)(n)=f(n+1)$, for any $f\in l^{\infty}(\bN)$ and $n\in \bN$. \end{definition}

Throughout the paper, we shall assume that all abelian C*-algebras are unital. From the viewpoint of dynamical systems, an anqie of $\mathbb{N}$ is a topological dynamical system associated with the additive structure of $\mathbb{N}$. Precisely, suppose that $\mathcal{A}$ is a C*-subalgebra of $l^{\infty}(\mathbb{N})$ and $X$ is its maximal ideal space. Let $\iota$ be the map from $\mathbb{N}$ to $X$ given by the multiplicative state of point evaluation, i.e., $\iota(n): f\mapsto f(n)$ for any $f\in \mathcal{A}$. We conclude that (see Proposition \ref{anqie is a shif invariant algebra})  $\mathcal{A}$ is an anqie if and only if the map $\iota(n)\mapsto \iota(n+1)$ from $\iota(\mathbb{N})$ to $X$ (corresponding to $n\mapsto n+1$ on $\mathbb{N}$) can be continuously extended to a map from $X$ to $X$. If we denote this map by $A$, then we also call $(X,A)$ \emph{an anqie $($of $\mathbb{N}$$)$}.

Based on such connections between arithmetics and dynamics, algebraic and analytical methods can be used in the study of arithmetic functions. This idea was explored in \cite{Wei18} to give some partial results about Sarnak's M\"{o}bius disjointness conjecture. In this paper, we further study arithemetic functions through anqies.

\subsection{Anqie entropy} A major tool we use to study anqies is the notion of ``entropy''. Originated in physics, ``entropy'' measures some uncertainty or disorderness of a thermodynamical system. A mathematical notion of entropy, initiated by Shannon \cite{Sha48}, is a measurement of information contained in signals or random variables. Inspired by Shannon's entropy, Kolmogorov \cite{Kol58} and Sinai \cite{Sin59} introduced a metric entropy on dynamical systems. Similar ideas were used by Adler-Konheim-McAndrew \cite{Adl-Kon-And65} (and later by Bowen \cite{Bow71}) to define a topological entropy for continuous maps on compact Hausdorff spaces. Moreover, Furstenberg \cite{Fur} introduced an entropy for stationary sequences of finite-valued random variables through the information contained in certain fields of measurable sets determined by them; Voiculescu \cite{Voi93} established free probability theory and introduced free entropy.

In this paper, we shall introduce the \emph{anqie entropy} of an anqie $\mathcal{A}$, denoted by $\AE(\mathcal{A})$, which is defined to be the topological entropy of the additive map $A$ on $X$. We refer readers to Section \ref{anqie entropy} for more details. For a family of arithmetic functions $\cF\subseteq l^\infty(\bN)$, denote by $\cA_\cF$ the anqie generated by $\cF$. It is the smallest $\sigma_{A}$-invariant C*-subalgebra of $l^\infty(\bN)$ containing $\cF$. For simplicity, we write $\AE(\cF)=\AE(\cA_\cF)$ and call it the \emph{anqie entropy} of $\cF$. When $\cF=\{f_1,\ldots,f_n\}$ is a finite set, we also use the notations $\cA_{f_1,\ldots,f_n}$ and $\AE(f_1,\ldots,f_n)$.

The following inequalities about anqie entropy will be proved in Section \ref{properties of anqie entropy},
\[\AE(f+g),\AE(f\cdot g)\leq \AE(f)+
\AE(g),\quad \AE(\sqrt{|f|})\leq \AE(f),\quad f,g\in l^\infty(\mathbb{N}).\]
More generally, the following sub-additivity holds.

\begin{theorem}\label{thm_entropy_additive} For any arithmetic functions
$f_1,\ldots,f_n,g_1,\ldots,g_m\in l^\infty(\bN)$, we have
$$\AE(f_1,\ldots,f_n,g_1,\ldots,g_m)\le \AE(f_1,\ldots,f_n)+
\AE(g_1,\ldots,g_m).$$ The above equality holds if $f_1,\ldots,f_n$
and $g_1,\ldots,g_m$ are two anqie independent families.
\end{theorem}

\noindent The formal definition of \emph{anqie independence} will be given in Section \ref{properties of anqie entropy}. One of the most crucial features of Shannon's entropy is its additivity for independent random variables. Theorem \ref{thm_entropy_additive} shows that the anqie entropy shares similar property as well.

Entropy sometimes has continuity properties. For example, Yomdin \cite{Yom1} and Newhouse \cite{New} showed that, for any compact smooth manifold $M$, the topological entropy function $f \mapsto h(f)$ from $C^\infty(M,M)$ to $[0,+\infty]$ is upper semi-continuous. Newhouse also concluded the continuity for $C^{\infty}$ diffeomorphisms of surfaces from
a result of Katok \cite{Kat}. Furthermore, for a compact interval $I$, Propositions 30 and 31 in Chapter 8 of \cite{Blo-Cop} show that the set of all maps $f\in C(I,I)$ with $h(f)=\infty$ is dense in $C(I,I)$, and the topological entropy, regarded as a map $h:\, C(I,I)\rightarrow [0,\infty]$, is lower semi-continuous. Also, Kolmogorov-Sinai entropy is an upper semi-continuous function of invariant measures. For the anqie entropy, we have the following result.

\begin{theorem}\label{infinity entropy is dense in the set of all functions}
Let $\mathcal{E}_{\infty}(\mathbb{N})$ be the set of all functions $f$ in $l^{\infty}(\mathbb{N})$ with $\AE(f)=+\infty$. Then $\mathcal{E}_{\infty}(\mathbb{N})$ is dense in $l^{\infty}(\mathbb{N})$ in norm topology.
\end{theorem}

The following theorem shows the lower semi-continuity of the anqie entropy.
\begin{theorem} \label{thm_semi_continuous}
Let $\cF=\{f_{0},f_{1},\ldots\}$ be a family of bounded arithmetic functions with $\AE(\mathcal{F})<+\infty$. Then for any $\epsilon>0$, there is a $\delta>0$, such that whenever $\cG=\{g_{0},g_{1},\ldots\}\subseteq l^\infty(\bN)$ satisfies $\sup_{i\geq 0}\|g_i-f_i\|_{l^\infty}<\delta$, we have ${\AE}(\cG)>{\AE}(\cF)-\epsilon$.
\end{theorem}

Intuitively, Theorem \ref{infinity entropy is dense in the set of all functions} shows that a slight perturbation on a function may turn its entropy to infinity, while Theorem \ref{thm_semi_continuous} implies that to substantially decrease the entropy of an arithmetic function, a large perturbation is needed.

\subsection{Arithmetic compactification of natural numbers} As a corollary of Theorem \ref{thm_semi_continuous}, we shall prove in Section \ref{the arithmetic compactification of natural numbers} that $\cE_0(\bN)$, the set of all bounded arithmetic functions with zero anqie entropy, is an anqie (see Theorem \ref{thm_C_0_subCalg}). We call the maximal ideal space of $\cE_0(\bN)$ \emph{the arithmetic compactification} of natural numbers, denoted by $E_{0}(\mathbb{N})$. Problems involving transitive topological dynamical systems with zero topological entropy may be studied in this space. For example, in our language, Sarnak's M\"{o}bius disjointness conjecture (see \cite{Liu-Sar15}) is equivalent to that $\lim\nolimits_{N\rightarrow \infty}N^{-1}\sum\nolimits_{n=0}^{N-1}\mu(n)f(n)=0$ for all continuous function $f$ on $E_0(\bN)$. Here $\mu(n)$ is the M\"obius function defined to be $(-1)^r$ if $n$ is the product of $r$ distinct primes and $0$ otherwise.

In Section \ref{the arithmetic compactification of natural numbers}, we shall study properties of $E_{0}(\mathbb{N})$. From the topological viewpoint, the arithmetic compactification is the maximal zero entropy topological factor of the Stone-\v{C}ech compactification of $\mathbb{N}$. Here, recall that the maximal ideal space of $l^{\infty}(\mathbb{N})$ is known as the Stone-\v{C}ech compactification of $\mathbb{N}$ (see \cite{Cec37}). The arithmetic and Stone-\v{C}ech compactifications of $\mathbb{N}$ are both uncountable and unmetrizable. For comparison, the Stone-\v{C}ech compactification is extremely disconnected (i.e., the closure of any open set is still open) and totally disconnected (i.e., each pair of points can be separated by sets that are both open and closed: \emph{clopen sets}), while for the arithmetic compactification, we obtain

\begin{theorem}\label{arithmetic compactification is totally disconnected}
The space $E_{0}(\mathbb{N})$ is totally disconnected but not extremely disconnected.
\end{theorem}

To further study $\cE_0(\bN)$, we investigate its $K$-groups. The operator $K$-theory provides us with useful tools to learn about the structure of C*-algebras. The $K$-groups can be treated as invariants to distinguish two C*-algebras. 
For example, Elliott showed in \cite{GAE} that the $K_{0}$-group is a complete invariant for approximately finite-dimensional C*-algebras. Another important application of $K$-theory to C*-algebras was presented by Pimsner and Voiculescu \cite{PV}. They proved that the reduced C*-algebra of the free group of two generators has no projection other than $0$ and $1$. We also mention that, in the study of $C^\ast$-dynamics, the $K$-groups are essential for the classification of homomorphisms and dynamical systems (see, e.g., \cite{Lin07}). The $C^\ast$-algebras and the associated topological dynamical systems in this paper are different from those $C^\ast$-dynamics. However, the study of $K$-groups would be helpful. For $\cE_0(\bN)$, we have the following results.

\begin{theorem} \label{thm_K0_group}
The group $K_{0}(\cE_{0}(\mathbb{N}))$ is homeomorphic to the additive group
$\{f\in \cE_0(\mathbb{N}): f(\mathbb{N}) \subseteq \bZ\}$.
\end{theorem}

\begin{theorem} \label{thm_K1_group}
The group $K_{1}(\mathcal{E}_{0}(\mathbb{N}))$ is trivial.
\end{theorem}
As a corollary of Theorem \ref{thm_K1_group}, we have the following result.
\begin{corollary}\label{zero entropy exponential expression}
For any arithmetic function $f$ with zero anqie entropy and $|f(n)|=1$ for any $n\in \mathbb{N}$, there is a real-valued function $g(n)$ with $\AE(g)=0$ such that $f(n)=\exp(\i g(n))$.
\end{corollary}
In the process to prove the above results, we develop an approximation method for arithmetic functions. We shall prove the following theorem in Section \ref{Approximation method for arithmetic functions}.
\begin{theorem}\label{theorem_finite_appoximate}
Suppose that $f$ is a bounded arithmetic function with anqie entropy $\lambda$ $(0\leq \lambda<+\infty)$. Then for any $N\geq 1$, there is an arithmetic function $f_{N}$ with finite range, such that $\AE(f_{N}) \leq \lambda$ and $\|f_{N}-f\|_{l^{\infty}}\leq \frac{1}{N}$.
\end{theorem}
In the following, we give some applications of the above conclusion.

\subsection{Applications}
Let $\mathcal{F}_{0}(\mathbb{N})$ be the subset of $l^{\infty}(\mathbb{N})$ consisting of functions with
zero anqie entropy and finite ranges. Then as a simple corollary of Theorem \ref{theorem_finite_appoximate},
$\mathcal{F}_{0}(\mathbb{N})$ is dense in $\mathcal{E}_{0}(\mathbb{N})$ relative to the norm topology. Note that a
function in $\cF_0(\bN)$ belongs to the algebra generated by $\{0,1\}$-valued functions with zero anqie
entropy (see Proposition \ref{01 valued function}). One can deduce from Theorem \ref{theorem_finite_appoximate}
that Sarnak's M\"{o}bius disjointness conjecture is true if and only if
$\lim\nolimits_{N\rightarrow \infty}N^{-1}\sum\nolimits_{n=1}^N\mu(n)f(n)=0$ for any $\{0,1\}$-valued
arithmetic function with $\AE(f)=0$ (see also \cite[Lemmas 4.28, 4.29]{AKLR}, where a proof is given via weak* convergence of invariant measures).

We also generalize the anqie entropy of an arithmetic function to that of a map $f:\mathbb{N}\rightarrow X$, where $X$ is a compact Hausdorff space. Let $X_{f}$ be the closure of the set $\{(f(n),f(n+1),\ldots):n\in \mathbb{N}\}$ in $X^{\mathbb{N}}$ and $B_{f}$ the Bernoulli shift $B$ restricted to the space $X_{f}$. We define the \emph{anqie entropy} of $f$ to be the topological entropy $h(B_{f})$. Note that this definition coincides with the definition of the anqie entropy of arithmetic functions when $X\subseteq \mathbb{C}$ (see Theorem \ref{thm_dynamicalGN}). As an application of Theorem \ref{theorem_finite_appoximate}, we show the following approximation result for topological dynamical systems in Section \ref{Approximation method for arithmetic functions}.

\begin{proposition} \label{approximation result of orbit}
Let $(X,d)$ be a compact metric space and $T$ a continuous map on $X$ with topological entropy $\lambda$ ($\lambda\geq 0$). Then for any $x\in X$ and $\epsilon>0$, there is a map $f$ from $\mathbb{N}$ to the set $\{T^{n}x: n\in \mathbb{N}\}$ with finite range, such that the anqie entropy of $f$ is less than or equal to $\lambda$ and $\sup_{n}d(T^{n}x,f(n))<\epsilon$.
\end{proposition}

We remark that similar results can be proved for $\mathbb{Z}$ and $E_{0}(\mathbb{Z})$.  Note that the additive map $A$ is invertible on $E_{0}(\mathbb{Z})$, but not on $E_{0}(\mathbb{N})$.  So the two topological dynamical systems $(E_{0}(\mathbb{Z}),A)$ and $(E_{0}(\mathbb{N}),A)$ are different. Many questions for $E_{0}(\mathbb{Z})$ may have easier answers than those for $E_{0}(\mathbb{N})$. Moreover, we may ask whether the similar results in this paper hold for other amenable topological semigroups.

This paper is organized as follows. In Section 2, we list some frequently used  notation and prove some preliminary results. In Section 3, we introduce the definition of anqie entropy. Some basic properties of the anqie entropy are also discussed. Comparisons of the anqie entropy with Shannon's entropy are presented in Section \ref{properties of anqie entropy}. Theorems \ref{thm_entropy_additive}, \ref{infinity entropy is dense in the set of all functions}, and \ref{thm_semi_continuous} are proved in this section. In Section \ref{the arithmetic compactification of natural numbers}, we investigate the structure of the arithmetic compactification $E_0(\bN)$ of $\bN$. We discuss the approximation method for maps from $\mathbb{N}$ to compact Hausdorff spaces in Section \ref{Approximation method for arithmetic functions}, where Theorems \ref{arithmetic compactification is totally disconnected}, \ref{theorem_finite_appoximate}, and Proposition \ref{approximation result of orbit} are proved. In Section \ref{kgroups}, we compute the $K_{0}$-group and $K_{1}$-group of the space of all continuous functions on $E_{0}(\mathbb{N})$. Corollary \ref{zero entropy exponential expression} is shown.

Some notions in this paper have been introduced by the survey paper of Liming Ge \cite{Ge}, in which the author announced that proofs of results there would appear later (see the introduction of \cite{Ge} and references there in). Sections 3, 4, 5 of this paper are based on results from the author's Ph.D. thesis written at the University of New Hampshire under the supervision of Professor Liming Ge. Results in Sections 6 and 7 are new. We refer to \cite{Conway,KR} for basics and preliminary results in operator algebra, to \cite{Gl,Wal82} for that in dynamical systems, and to \cite{B,RLL} for that in $K$-theory for C*-algebras.

\emph{Notation}. For an arithmetic function $f$, we use $\overline{f(\mathbb{N})}$ to denote the closure of $f(\mathbb{N})$ in $\mathbb{C}$. For a compact Hausdorff space $X$, denote by $X^{\mathbb{N}}$ the Cartesian product of $X$ indexed by $\mathbb{N}$. The topology on $X^{\mathbb{N}}$ we concern in this paper is always the product topology. For a finite set $C$, the notation $|C|$ means the cardinality of $C$. For a subset $R$ of $\bN$, we write $\chi_R$ for the characteristic function defined on $R$. The symbol $\i$ denotes the imaginary unit $\sqrt{-1}$.

\bigskip

\textbf{Acknowledgments}. I am very grateful to my advisor Professor Liming Ge for encouraging me in this research and for his guidance. I would like to thank Professor Arthur Jaffe for his support; Dr. Boqing Xue and Weichen Gu for their valuable discussions. I heartly thank Professor Jinxin Xue for his very helpful comments and suggestions on the manuscript. This research was supported in part by the University of New Hampshire, by the Academy of Mathematics and Systems Science of the Chinese Academy of Sciences and by Grant TRT 0159 from the Templeton Religion Trust, and by the fellowship of China Postdoctoral Science Foundation 2020M670273.

\section{Preliminaries}
In this section, we prove some preliminary results. First, we list some notation that will be used.

Let $\mathcal{H}$ be a Hilbert space. Denote by $\mathcal{B}(\mathcal{H})$ the algebra consists of all bounded linear operators on $\mathcal{H}$.
By Riesz representation theorem, for any $T\in\mathcal{B}(\mathcal{H})$,
there is a unique bounded linear operator $T^{*}$ satisfying $\langle Tx,y\rangle=\langle x,T^{*}y\rangle$ for any $x,y\in\mathcal{H}$.
Such a $T^*$ is called the \emph{adjoint of $T$}. We call a norm-closed *-subalgebra of $\mathcal{B}(\mathcal{H})$ a \emph{C*-algebra}.

Suppose that $\mathcal{A}$ is a unital C*-algebra. We use $\mathcal{A}^{\sharp}$ to denote the set of all bounded linear functionals on $\mathcal{A}$. Denote by $(\mathcal{A}^{\sharp})_{1}$ the unit ball in $\mathcal{A}^{\sharp}$, i.e., $(\mathcal{A}^{\sharp})_{1}=\{\rho\in \mathcal{A}^{\sharp}:\|\rho\|\leq 1\}$. In general, the space $\mathcal{A}^{\sharp}$ can be equipped with many topological structures. Among them, the norm topology and weak* topology are used most frequently. For $\rho\in \mathcal{A}^{\sharp}$, its norm is given by $\|\rho\|=\sup_{x\in \mathcal{A},\|x\|\leq 1}|\rho(x)|$. When $x\in \mathcal{A}$, the equation $\sigma_{x}(\rho)=|\rho(x)|$ defines a semi-norm on $\mathcal{A}^{\sharp}$. The family $\{\sigma_{x}: x\in \mathcal{A}\}$ of semi-norms determines the \emph{weak* topology} on $\mathcal{A}^{\sharp}$. Note that each $\rho_{0}\in \mathcal{A}^{\sharp}$ has a base of neighborhoods consisting of sets of the form $\{\rho\in \mathcal{A}^{\sharp}: |\rho(x_{j})-\rho_{0}(x_{j})|<\epsilon\}$ ($j=1,\ldots,m$), where $\epsilon>0$ and $x_{1},\ldots,x_{m}\in \mathcal{A}$.

A non-zero linear functional $\rho$ on an abelian C*-algebra $\A$ is called \emph{a multiplicative state} if for any $A,B\in \mathcal{A}$, $\rho(AB)=\rho(A)\rho(B)$.

Suppose now that $\mathcal{A}$ is an abelian C*-algebra and $X$ is its maximal ideal space. We define the map $\gamma: \mathcal{A}\rightarrow C(X)$ by
\begin{equation}\label{Gelfand transform}
\gamma(f)(\rho)=\rho(f), \quad f\in \cA, \,\rho \in X.
\end{equation}
Here we use the fact that $X$ is also the space of all multiplicative states of $\mathcal{A}$. The map $\gamma$ is known as the \emph{Gelfand transform} from $\mathcal{A}$ onto $C(X)$, which is a *-isomorphism (see, e.g., \cite[Theorem 2.1]{Conway}).

It is known that the above Hausdorff space $X$ is weak* compact. Next we show that $\mathcal{A}$ is countably generated as an abelian C*-algebra if and only if $X$ is metrizable and the topology induced by the metric coincides with the weak* topology on $X$. The sufficient part directly follows from \cite[Remark 3.4.15]{KR}. The necessary part is showed in the following proposition.

\begin{proposition}\label{metrizable}
Let $\mathcal{A}$ be an abelian unital C*-algebra. If $\mathcal{A}$ is countably generated, then $(\mathcal{A}^{\sharp})_{1}$ is metrizable and the toplology induced by the metric is equivalent to the weak* topology on $(\mathcal{A}^{\sharp})_{1}$. In particular,
the maximal ideal space of $\mathcal{A}$ is a compact metrizable space.
\end{proposition}
\begin{proof}
Since $\mathcal{A}$ is countably generated, there is a countable dense subset in $\mathcal{A}$.
Let $\{g_{1},g_{2},\ldots\}$ be a dense subset of $(\mathcal{A})_1$, the unit ball in $\mathcal{A}$. For any $\rho_{1},\rho_{2}\in (\mathcal{A}^{\sharp})_{1}$,
we define $d(\rho_{1},\rho_{2})=\sum_{i=1}^{\infty}\frac{|(\rho_{1}-\rho_{2})(g_{i})|}{2^{i}}$.
It is not hard to check that $d$ is a metric on $(\mathcal{A}^{\sharp})_{1}$.
Moreover, for any net $\{\rho_{\alpha}\}$ of elements of $(\mathcal{A}^{\sharp})_{1}$, the net $\{d(\rho_{\alpha},\rho)\}$ converges to $0$ is equivalent to the condition that,
for any $i\geq 1$, the net $\{\rho_{\alpha}(g_{i})\}$ converges to $\rho(g_{i})$.

Next, we show that the weak* topology is equivalent to the topology induced by the metric $d$ on $(\mathcal{A}^{\sharp})_{1}$.
Suppose that the net $\{\rho_{\alpha}\}$ of elements in $(\mathcal{A}^{\sharp})_{1}$, weak* converges to $\rho$. Then, for any $i\geq 1$,
the net $\{\rho_{\alpha}(g_{i})\}$ converges to $\rho(g_{i})$. Thus the net $\{d(\rho_{\alpha},\rho)\}$ converges to $0$.
Conversely, if the net $\{d(\rho_{\alpha},\rho)\}$ converges to $0$, where $\rho_{\alpha}\in (\mathcal{A}^{\sharp})_{1}$, then $\{\rho_{\alpha}(g_{i})\}$ converges to $\rho(g_{i})$ for any $i\geq 1$.
Note that, for any $\alpha$, $\|\rho_{\alpha}\|\leq 1$. Then for any $g\in \mathcal{A}$, the net
$\{\rho_{\alpha}(g)\}$ converges to $\rho(g)$. So the net $\{\rho_{\alpha}\}$ is weak* convergent to $\rho$ in $(\mathcal{A}^{\sharp})_{1}$.

By Alaoglu-Bourbaki theorem $(\mathcal{A}^{\sharp})_{1}$ is weak* compact. Let $X$ be the maximal ideal space of $\mathcal{A}$.
Then, relative to the weak* topology, $X$ is a closed subset of $(\mathcal{A}^{\sharp})_{1}$. From the above analysis, we see that the weak* topology on $(\mathcal{A}^{\sharp})_{1}$ coincides with the topology induced by the metric $d$ on it. Thus $X$ is a compact metrizable space.
\end{proof}

\begin{proposition}\label{the set of natural numbers are dense in X}
Suppose that $\mathcal{A}$ is a C*-subalgebra of $l^{\infty}(\mathbb{N})$ and $X$ the maximal ideal space of $\mathcal{A}$.
Let $\iota: \mathbb{N}\rightarrow X$ be the map given by $\iota(n):f \mapsto f(n)$, for any $f\in \mathcal{A}$. Then the weak* closure of $\iota(\mathbb{N})$ is $X$ $($write $\overline{\iota(\mathbb{N})} = X$$)$.
\end{proposition}

\begin{proof}
 Assume on the contrary that $\overline{\iota(\bN)}\neq X$. Choose $y\in X \setminus \overline{\iota(\bN)}$. By Urysohn's lemma, there is a $G\in C(X)$
 such that $G(y)=1$ and $G(x)=0$ for any $x\in \overline{\iota(\bN)}$. By equation (\ref{Gelfand transform}), for any $n\in \mathbb{N}$, $0=G(\iota(n))=\iota(n)(\gamma^{-1}G)=(\gamma^{-1}G)(n)$. Then $\gamma^{-1}(G)=0$ and $G=0$ correspondingly. This contradicts $G(y)=1$. Hence $\overline{\iota(\bN)}=X$.
 \end{proof}

\begin{proposition}\label{anqie is a shif invariant algebra}
Suppose that $\mathcal{A}$ is a C*-subalgebra of $l^{\infty}(\mathbb{N})$ and $X$ the maximal ideal space of $\mathcal{A}$. Then $\mathcal{A}$ is an anqie of $\mathbb{N}$ if and only if the map $\iota(n)\mapsto \iota(n+1)$ can be extended to a continuous map from $X$ to itself.
\end{proposition}

\begin{proof}
Suppose that the map $\iota(n)\mapsto \iota(n+1)$ is extended to a continuous map on $X$, denoted by $A$. Given $f\in \mathcal{A}$, assume that $F=\gamma(f)$ (see equation (\ref{Gelfand transform})). Note that $F\circ A\in C(X)$. Let $g= \gamma^{-1}(F\circ A)$ in $\mathcal{A}$. Then $g(n)=F\circ A(\iota(n))=F(\iota(n+1))=f(n+1)=\sigma_{A}f(n)$. Thus $g=\sigma_{A}f$ in $\mathcal{A}$. This shows that $\mathcal{A}$ is $\sigma_{A}$-invariant and thus an anqie of $\mathbb{N}$.

On the other hand, suppose that $\mathcal{A}$ is an anqie of $\mathbb{N}$. Let $A$ be the map from $X$ to itself given by $A\rho(f)=\rho(\sigma_{A}f)$ for any $\rho\in X$ and $f\in \mathcal{A}$.
It is easy to see that $A(\iota(n))=\iota(n+1)$.
Now we show that $A$ is a continuous map on $X$. If $\{\rho_{\alpha}\}$ is a weak* convergent net of elements of $X$, with limit $\rho$, then
for any $f\in \A$, $\rho_{\alpha}(\sigma_{A}f)=A\rho_{\alpha}(f)$ converges to $\rho(\sigma_{A}f)=A\rho(f)$.
Thus the net $\{A\rho_{\alpha}\}$ weak* converges to $A\rho$ in $X$. Hence $A$ is the continuous map on $X$ extended by $\iota(n)\mapsto \iota(n+1)$.
\end{proof}

At the end of this section, we recall Weyl's Criterion (see, e.g., \cite[Chapter 21]{Iwa-Kow}) which will be used in this paper to compute some examples. A sequence $\{\alpha_{n}=(x_{n,1},\ldots,
x_{n,k})\}_{n=1}^{\infty}$ in $\mathbb{R}^{k}$ is said to be \emph{uniformly distributed modulo $1$} if for any $[a_j, b_j]\subseteq [0,1]$, $j=1,\ldots, k$, we have $
\lim_{N\rightarrow \infty}N^{-1}\sum_{n=1}^{N}\prod_{j=1}^{k}\chi_{[a_{j},b_{j}]}(\{x_{n,j}\})=\prod_{j=1}^k (b_j-a_j)$, where $\{x\}$ denotes the fractional part of the real number $x$.

\begin{proposition}[Weyl's Criterion]\label{uniformly distributed}
The sequence $\{\alpha_{n}=(x_{n,1},\ldots,
x_{n,k})\}_{n=1}^{\infty}$ in $\mathbb{R}^{k}$ is uniformly distributed modulo $1$ if and only if for any $(l_1,\ldots, l_k)\in \mathbb{Z}^k\setminus \{0\}$,

$$
\lim_{N\rightarrow \infty}\frac1N\sum_{n=1}^N e^{2\pi i (l_1x_{n,1}+\cdots
+l_kx_{n,k})}=0.
$$
\end{proposition}

We also need the following well-known result (see, e.g., \cite[Exercise 11.1.21]{Mur}).
\begin{lemma}\label{the polynomial is uniformly distributed}
Let $P(n)=a_{d}n^{d}+a_{d-1}n^{d-1}+\cdot\cdot\cdot+a_{1}n+a_{0}$ be a polynomial with real coefficients. Assume that at least one coefficient $a_{i}$
with $i\geq 1$ is irrational. Then the sequence of fractional parts of $P(n)$ is uniformly distributed modulo 1.
\end{lemma}

As a corollary of Proposition \ref{uniformly distributed} and Lemma \ref{the polynomial is uniformly distributed}, we have the following result.

\begin{corollary}\label{the square of n times an irrational number}
Suppose that $x_{n}=(\{n^2\theta\}, \{(n+1)^2\theta\})$ for $n\geq 1$,
where $\theta$ is irrational.
Then the sequence $\{x_{n}\}_{n=1}^{\infty}$ is uniformly distributed modulo 1.
\end{corollary}

\section{Anqie entropy}\label{anqie entropy}
From the definition of anqie, we see that each anqie corresponds to a topological dynamical system associated with the additive structure of $\mathbb{N}$. To understand anqies, it is natural to study the associated dynamics. The ``entropy" for anqies defined in this section is based formally on the topological entropy of the additive map $A$ (see the paragraph below Definition \ref{definition of anqie}). Here, let us first recall the definition of topological entropy for topological dynamical systems.
\begin{definition}\label{topological entropy}
Let $X$ be a compact Hausdorff space and $T$ a
continuous map on $X$. Suppose $\mathcal{U}$ and $\mathcal{V}$ are
two open covers for $X$. Denote by $\mathcal{U}\vee\mathcal{V}$ the
open cover whose elements consist of all intersections of elements from
$\mathcal{U}$ and $\mathcal{V}$ $($i.e.,
$\mathcal{U}\vee\mathcal{V}=\{A\cap B:
A\in\mathcal{U},B\in\mathcal{V}\}$$)$, and by $\mathcal{N}({\mathcal{U}})$
the minimal number of open sets in $\mathcal{U}$ that cover
$X$. Define
\begin{eqnarray*}
  h(T,\mathcal{U}) & =& \lim_n\frac{1}{n} \{\log\left(\mathcal{N}(\mathcal{U}\vee
  T^{-1}\mathcal{U}\vee\cdots\vee T^{-n+1}\mathcal{U})\right)\}, \\
  h(T)     & = &\sup_\mathcal{U}\left\{h(T,\mathcal{U}):\mathcal{U}\
  {\mathrm{is\ an \ open\ cover\ of }}\
  X\right\}.
\end{eqnarray*}
We call $h(T)$ the topological entropy of $T$.
\end{definition}

We refer to \cite{Adl-Kon-And65} for basics on topological entropy. Topological entropy is an invariant of topological dynamical systems. For two topological dynamical systems $(X_{1},T_{1})$ and $(X_{2},T_{2})$, if there is a continuous surjective map $\pi:X_{1}\rightarrow X_{2}$ such that $T_{2}\pi=\pi T_{1}$, we say that $(X_{2},T_{2})$ is a \emph{$($topological$)$ factor} of $(X_{1},T_{1})$ and $\pi$ a \emph{factor map}. In this case, we have $h(T_{2})\leq h(T_{1})$. Additionally, we say that $(X_{1},T_{1})$ and $(X_{2},T_{2})$ are \emph{topologically conjugate} (to each other) when $\pi$ is an homeomorphism. In this case, we have $h(T_{1})=h(T_{2})$.

When $(X,T)$ is a point transitive dynamical system and a transitive point $x_0$ is given, we shall denote this dynamical system as $(X,T,x_0)$. We say that $(X_2,T_2,x_2)$ is a factor of $(X_1,T_1,x_1)$ if there is a factor map $\pi:X_{1}\rightarrow X_{2}$ such that $\pi(x_1)=x_2$. When $\pi$ is a homeomorphism, we say that $(X_1,T_1,x_1)$ and $(X_2,T_2,x_2)$ are equivalent (to each other).

Now we are ready to introduce the notion of entropy for anqies.

\begin{definition}
Suppose that
$\A\subseteq l^\infty(\bN)$ is an anqie. We define the
\textbf{anqie entropy} of $\A$, denoted by $\AE(\mathcal{A})$, to be the topological entropy $h(A)$ of the
additive map $A$ which extends the map $n\mapsto n+1$ on $\bN$ to the whole maximal
ideal space of $\A$. If $\A$ is generated by a family of bounded arithmetic functions $\mathcal{F}$
as an anqie, we call $\AE(\mathcal{A})$ the \textbf{anqie entropy of $\mathcal{F}$}, denoted by $\AE(\mathcal{F})$. In particular, when $\mathcal{F}=\{f_{0},f_{1},\ldots\}$ is at most countable, we also use the notation $\AE(f_{0},f_{1},\ldots)$ to denote the anqie entropy of $\mathcal{F}$.
\end{definition}

We list some simple but very useful facts of $\AE$ in the following
lemma.

\begin{lemma}\label{basic property} Suppose that $\A_{1}, \A_{2}\subseteq l^\infty(\bN)$ are anqies.

(\romannumeral1) If $\A_{1}$ is a subanqie of $\A_{2}$, i.e., $\A_{1} \subseteq \A_{2}$, then $\AE(\A_{1})\le \AE(\A_{2})$;

(\romannumeral2) For any $f_1,\ldots,f_n\in l^\infty(\bN)$ $(n\geq 1)$ and any
polynomials $\phi_1,\ldots,\phi_{m}\in \mathbb{C}[x_1,\ldots,x_n]$ $(m\geq 1)$, we have
$
\AE(\phi_1(f_1,\ldots,f_n),\ldots,\phi_m(f_1,\ldots,f_n))\le
\AE(f_1,\ldots,$ $f_n)$.
\end{lemma}

\begin{proof}
(\romannumeral1) Suppose that $(X_{1},A_{1})$ and $(X_{2},A_{2})$ are topological dynamical systems corresponding to $\mathcal{A}_{1}$ and $\mathcal{A}_{2}$, respectively. Let $\phi: X_2\rightarrow X_1$ be the map given by
\begin{equation} \label{eq_induced_between_A_1_and_A_2}
(\phi(\omega)) (f) = \omega(f),\quad \omega \in X_2,\, f\in \cA_1.
\end{equation}
Since every maximal ideal in $\mathcal{A}_{1}$ extends to a maximal ideal in $\mathcal{A}_{2}$, we have that $\phi$ is surjective. The continuity of $\phi$ follows from the definition of weak* topology on $X_{1}$ and $X_{2}$. Moreover, it is not hard to check that $\phi\circ A_{2}=A_{1}\circ\phi$. Then $(X_{1},A_{1})$ is a topological factor of $(X_{2},A_{2})$. Hence $h(A_{1})\leq h(A_{2})$ and $\AE(\mathcal{A}_{1})\leq \AE(\mathcal{A}_{2})$.

(\romannumeral2) follows from (\romannumeral1).
\end{proof}

More generally, the above polynomials $\phi_j$'s can even be replaced by continuous functions defined on the maximal ideal space of the anqie generated by $f_1,\ldots,f_n$. For example, we can obtain $\AE(\text{Re}(f)),\AE(\text{Im}(f)),\AE(\sqrt{|f|})\leq \AE(f)$ for any $f\in l^\infty(\bN)$.  Clearly, the anqie entropy $\AE(\mathcal{F})$ takes value in $[0,+\infty]$ for any family $\mathcal{F}\subseteq l^{\infty}(\mathbb{N})$. One may ask whether there is an arithmetic function with infinite anqie entropy. In the following, we construct an $f$ so that $\AE(f)=\infty$. Then from Lemma \ref{basic property}(i), one can easily see that $\AE(l^{\infty}(\mathbb{N}))=\infty$.

\begin{example}\label{an example of infinity entropy}
{\rm
 Let $\cup_{n\geq 0}$ $[1-2^{-n},1-2^{-n-1})$ be a partition of $[0,1)$. For each $n$, there is a partition of $[1-2^{-n},1-2^{-n-1})$ into $2^{n^2}$ subintervals of equal length, denote them by $[x_{n,i},x_{n,i+1})$ accordingly, for $i=0,1,\ldots,2^{n^2}-1$. On the interval $[x_{n,i},x_{n,i+1}]$, define $F((x_{n,i}+x_{n, i+1})/2)=2^{-n/2}$  and $F(x_{n,i})=F(x_{n,i+1})=0$, with remaining part connected linearly $($assume $F(1)=0$$)$. Then $F$ is a zigzag function of height $2^{-n/2}$ in the interval $[x_{n,i}$, $x_{n,i+1}]$. Moreover, $F$ is continuous on $[0,1]$. One can verify that, for any open subset $U$ of $[0,1]$, there is an integer $N$ such that $F^{N}(U)=[0,1]$ $($Here denote by $F^{N}$ the composition of $F$ for $N$ times$)$. Hence the topological dynamical system $([0,1],F)$ is transitive and there is an $x_{0}\in [0,1]$ such that the set $\{F^{n}x_{0} : n\in \mathbb{N}\}$ is dense in $[0,1]$ $($see \cite[Chapter \uppercase\expandafter{\romannumeral5}, Proposition 39]{Blo-Cop}$)$.

 Note that for $m\geq 0$, when $x\in [0,2^{-2m}]$, $F^{m}(x)=4^{m}x$. Choose the interval $J_{n,i}=[x_{n,i}/2^{2m},(x_{n,i}+x_{n,i+1})/2^{2m+1}]$ for $n=0,1,\ldots,4m$ and $i=0,1,\ldots,2^{n^2}-1$. So there are $1+2+2^4+\cdot\cdot\cdot+2^{(4m)^2}$ $(=a_m)$ disjoint closed intervals $J_{n,i}$ in $[0,2^{-2m})$ such that for each $n'$, $i'$ with $0\leq n'\leq 4m$ and $0\leq i'\leq 2^{n^2}-1$,
$
\cup_{n=0}^{4m}\cup_{i=0}^{2^{n^2}-1}J_{n,i}\subseteq [0,2^{-2m}) \subseteq F^{m+1}(J_{n',i'})=F([x_{n',i'},(x_{n',i'}+x_{n',i'+1})/2])=[0,2^{-n'/2}].
$
By \cite[Chapter\uppercase\expandafter{\romannumeral8}, Proposition 8]{Blo-Cop}, $h(F^{m+1})\geq \log a_{m}>\log 2^{(4m)^2}$. Hence, we obtain
 \[h(F)=\frac{h(F^{m+1})}{m+1}> \frac{\log 2^{(4m)^2}}{m+1}
\rightarrow \infty \hspace{.2cm} (\textmd{when}~m\rightarrow \infty).\]
Therefore $h(F)=\infty$. Now we let $f(n)=F^{n}(x_{0})$ and $X_{f}$ the maximal ideal space of $\mathcal{A}_{f}$. Then $(X_{f},A)$ and $([0,1],F)$ are topologically conjugate. Thus $\AE(f)=h(F)=\infty$.
}
\end{example}

In fact, the set $\mathcal{E}_{\infty}(\mathbb{N})$ consisting of bounded arithmetic functions $f$ with $\AE(f)=\infty$ is dense in $l^{\infty}(\mathbb{N})$ (see Theorem \ref{infinity entropy is dense in the set of all functions}), which will be proved in the next section. Now let us explore a way to compute the anqie entropy of a family of arithmetic functions $\mathcal{F}$. From the definition of anqie entropy, we see that it is important to know the properties of the anqie associated with $\mathcal{F}$.

Let $\mathcal{F}$ be a family of bounded arithmetic functions. Suppose that $\mathcal{A}_{\mathcal{F}}$ is the anqie generated by $\mathcal{F}$, i.e., $\cA_{\cF}$ is the C*-algebra generated by $\{1,(\sigma_{A})^j f: \, f\in \cF, j\in \bN\}$. Let $X_{\mathcal{F}}$ be the maximal ideal space of $\mathcal{A}_{\mathcal{F}}$. Then $X_\cF$ is metrizable if $\cF$ is countable (see Proposition \ref{metrizable}). The following dynamical Gelfand-Naimark theorem gives a description of $X_{\mathcal{F}}$. This provides us with a method to compute the anqie entropy of $\mathcal{F}$.

\begin{theorem} \label{thm_dynamicalGN}
Let $\cF=\{f_\lambda\}_{\lambda\in \Lambda}$ be a family of bounded arithmetic functions, where $\Lambda$ is an index set. Let $\cA_0$ be the C*-algebra generated by $\cF$, and $\cA_\cF$ the smallest $\sigma_{A}$-invariant C*-algebra that contains $\cA_0$. Suppose that $X_{\mathcal{F}}$ is the maximal ideal space of $\mathcal{A}_{\mathcal{F}}$. The following statements hold.

(\romannumeral1) For each $n\in \bN$, write $z_n=\left(f_\lambda(n)\right)_{\lambda\in\Lambda}$, an element in $\prod\nolimits_{\lambda\in \Lambda}\overline{f_\lambda(\bN)}$. Let $X_0$ be the closure of $\{z_n: n\in \bN\}$ in $\prod\nolimits_{\lambda\in \Lambda}\overline{f_\lambda(\bN)}$. Then $\cA_0\cong C(X_0)$.

(\romannumeral2) Write $\kappa_n=(z_n,z_{n+1},\ldots)$, an element in $X_0^{\bN}$. Let $Y_\cF$ be the closure of $\{\kappa_n:\, n\in \bN\}$ in $X_0^\bN$. Then $\cA_{\mathcal{F}} \cong C(Y_\cF)$.

(\romannumeral3) Assume that $B$ is the Bernoulli shift on $X_{0}^{\mathbb{N}}$ defined by $B: (\omega_0,\omega_1,\ldots) \mapsto (\omega_1,\omega_2,\ldots)$. Then the map $B$ restricted to $Y_{\mathcal{F}}$, denoted by $B_{\mathcal{F}}$, is identified with $A$ on $X_{\mathcal{F}}$.
\end{theorem}

\begin{proof}
(\romannumeral1) Define the map $\varpi:C(X_0)\rightarrow l^\infty(\bN)$ by $\varpi(g)(n)=g(z_n)$ for $g\in C(X_0)$. Let $\widetilde{\phi_\lambda}: \prod\nolimits_{\lambda\in \Lambda}\overline{f_\lambda(\bN)}\to \overline{f_\lambda(\bN)}$ be the
projection map from $\prod\nolimits_{\lambda\in \Lambda}\overline{f_\lambda(\bN)}$ onto its $\lambda$-th coordinate. Denote by $\phi_{\lambda}$ the restriction of $\widetilde{\phi_\lambda}$ to $X_{0}$. Then $\phi_\lambda\in C(X_{0})$ and $\varpi(\phi_{\lambda})=f_{\lambda}$ for any $\lambda\in \Lambda$. Since $\{z_n: n\in \bN\}$ is dense in $X_{0}$, we conclude that $\varpi$ is isometric and $\mathcal{A}_{0}\subseteq \varpi(C(X_{0}))$.

Let $\omega_{1}=(\omega_{1,\lambda})_{\lambda\in \Lambda}$ and $\omega_{2}=(\omega_{2,\lambda})_{\lambda\in \Lambda}$ be distinct points in $X_{0}$. Then there is a $\lambda_{0}\in \Lambda$ such that $\omega_{1,\lambda_{0}}\neq \omega_{2,\lambda_{0}}$. Hence $\phi_{\lambda_{0}}=\varpi^{-1}(f_{\lambda_{0}})$ satisfies $\phi_{\lambda_{0}}(\omega_{1})\neq\phi_{\lambda_{0}}(\omega_{2})$.  This implies that $\varpi^{-1}(\cA_{0})$ separates the points of $X_{0}$. Note that $\varpi^{-1}(\mathcal{A}_{0})$ is a closed *-subalgebra of $C(X_{0})$. By the Stone-Weierstrass theorem (see, e.g., \cite[Theorem 3.4.14]{KR}), $\varpi^{-1}(\mathcal{A}_{0})=C(X_{0})$. Thus $\cA_{0}\cong C(X_{0})$.

(\romannumeral2) Define the map $\pi: C(Y_{\mathcal{F}})\rightarrow l^\infty(\bN)$ by $\pi(h)(n)=h(\kappa_n)$
for $h\in C(Y_{\mathcal{F}})$. Note that $\{\kappa_n:\, n\in \bN\}$ is dense in $Y_{\mathcal{F}}$, then $\pi$ is an injective homomorphism and $C(Y_{\mathcal{F}})$ is *-isomorphic to $\pi(C(Y_{\mathcal{F}}))$.

To prove $\cA_{\mathcal{F}}\subseteq \pi(C(Y_{\mathcal{F}}))$, we only need to show that, for any given $j\in \bN$ and $f\in \mathcal{A}_{0}$, $(\sigma_{A})^jf\in \pi(C(Y_{\mathcal{F}}))$. Define a function $\varphi_{j}$ on $Y_{\mathcal{F}}$, such that for any $\rho=(\rho_{0},\rho_{1},\ldots)\in Y_{\mathcal{F}}$,
\begin{equation}\label{define coordinate function}
\varphi_{j}(\rho)=f(\rho_{j}).
\end{equation}
Here we use the fact that $f$ corresponds to $f(x)\in C(X_{0})$ by (\romannumeral1). Since $\rho_{j}$ is the image of the projection of $\rho$ to the $(j+1)$-th coordinate, one concludes that $\varphi_{j}\in C(Y_{\mathcal{F}})$. From $\varphi_{j}(\kappa_{n})=f(z_{n+j})=f(n+j)$, we obtain $\pi(\varphi_{j})=(\sigma_{A})^{j}f$. Thus $\mathcal{A}_{\mathcal{F}}\subseteq \pi(C(Y_{\mathcal{F}}))$.

Given distinct points $\rho=(\rho_{0},\rho_{1},\ldots)$ and $\rho'=(\rho_{0}',\rho_{1}',\ldots)\in Y_{\mathcal{F}}$, there is an $l\in \mathbb{N}$ such that $\rho_{l}\neq \rho'_{l}$ in $X_{0}$. Then by (\romannumeral1), there is an $f_{1}\in \mathcal{A}_{0}$ so that $f_{1}(\rho_{l})\neq f_{1}(\rho'_{l})$. Let $\varphi_{l}$ be the function defined in equation (\ref{define coordinate function}) replacing $j$ by $l$ and $f$ by $f_{1}$. Then $\varphi_{l}(\rho)\neq \varphi_{l}(\rho')$. Combining with $\varphi_{l}=\pi^{-1}((\sigma_{A})^{l}f_{1})$, we conclude that $\pi^{-1}(\cA_{\mathcal{F}})$ separates points in $Y_{\mathcal{F}}$. Note that $\pi^{-1}(\mathcal{A}_{\mathcal{F}})$ is a closed *-subalgebra of $C(Y_{\mathcal{F}})$. By the Stone-Weierstrass theorem, $\pi^{-1}(\mathcal{A}_{\mathcal{F}})=C(Y_{\mathcal{F}})$. Thus $\cA_{\mathcal{F}}\cong C(Y_{\mathcal{F}})$.

(\romannumeral3) Recall that the map $A$ on $X_{\mathcal{F}}$ sends $\iota(n)$, the multiplicative state of point evaluation at $n$, to $\iota(n+1)$. From the definition of $\pi$ in the above proof of (\romannumeral2), it is not hard to check that the map $X_{\mathcal{F}}\rightarrow Y_{\mathcal{F}}$ extended by $\iota(n)\mapsto \kappa_n$, denoted as $F$, is a homeomorphism. Note that $B_{\mathcal{F}}\circ F=F\circ A$. Thus the additive map $A$ on $X_{\mathcal{F}}$ coincides with the restriction of $B$ to $Y_{\mathcal{F}}$.
\end{proof}

In the following, we treat $X_{\mathcal{F}}$ the same as $Y_{\mathcal{F}}$ and call $(X_\cF,B_{\mathcal{F}})$ the \textit{canonical representation} of the anqie $\cA_\cF$. By the definition of anqie entropy, we see that $\AE(\mathcal{F})=h(B_{\mathcal{F}})$. Next, we show an example to help us understand the above theorem.

\begin{example}\label{torus}
{\rm
Let $f(n)=e^{2\pi \i n^2\theta}$, for $n\geq 0$ and $\theta$ irrational. Denote by $\mathcal{A}_{f}$ the anqie generated by $f$ and $X_{f}$ the maximal ideal space of $\mathcal{A}_{f}$. Then $X_{f}$ is homeomorphic to $S^{1}\times S^{1}$ and $\AE(f)=0$, where $S^{1}$ is the unit circle.
}
\end{example}

The proof of the above fact is more involved. Here are some details.
Let $X_{0}=\overline{f(\mathbb{N})}$. Theorem \ref{thm_dynamicalGN} gives that $X_{f}$ is the closure of $\{\kappa_n=(e^{2\pi \i n^2\theta},e^{2\pi \i (n+1)^2\theta},\ldots):n\in \mathbb{N}\}$ in $X_{0}^{\mathbb{N}}$. It is easy to see that for any $x=(x_{0},x_{1},\ldots)\in X_{f}$, the coordinate $x_{l}$ ($l\geq 2$) can be determined by the first two coordinates $x_{0}$, $x_{1}$. In fact,
\begin{equation}\label{the thid coordinate}
x_{l}=e^{2\pi \i(l^2-l)\theta}x_{0}(x_{1}\overline{x_{0}})^l.
\end{equation}
Define $\Phi$ to be the projection from $X_{f}$ onto its first two coordinates, i.e., for any $x=(x_{0},x_{1},\ldots)\in X_{f}$, $\Phi(x)=(x_{0},x_{1})$. From the above analysis, we see that $\Phi$ is a homeomorphism from $X_{f}$ onto $\Phi(X_{f})$. By Corollary \ref{the square of n times an irrational number}, the sequence $\{(\{n^2\theta\},\{(n+1)^2\theta\})\}_{n=0}^{\infty}$ is uniformly distributed modulo 1. Then $\Phi(X_{f})=S^{1}\times S^{1}$ and thus $X_{f}$ is homeomorphic to $S^{1}\times S^{1}$.

In addition, the map $A$ (or the Bernoulli shift) on $X_{f}$ corresponds to a continuous map on $S^{1}\times S^{1}$, denoted by $A$ again, such that, for any $(x_{0},x_{1})\in S^{1}\times S^{1}$, $A(x_{0},x_{1})=(x_{1},x_{2})$, where $x_{2}=e^{4\pi \i \theta}x_{1}^2\overline{x_{0}}$ by equation (\ref{the thid coordinate}). If we identify $S^{1}\times S^{1}$ with $\mathbb{R}/\mathbb{Z}\times \mathbb{R}/\mathbb{Z}$, then we can rewrite the map $A$ as
 \begin{align*}
        A((\alpha_1, \alpha_2)) =
        \begin{pmatrix}
            0 & 1 \\
            -1 & 2\\
        \end{pmatrix}
        \begin{pmatrix}
           \alpha_1 \\
           \alpha_2 \\
        \end{pmatrix} +
        \begin{pmatrix}
           0 \\
           2\theta
        \end{pmatrix}.
    \end{align*}
The map $A$ is an affine linear transformation on the torus and the coefficient matrix only has eigenvalue $1$. By \cite[Theorems 8.11, 8.14]{Wal82}, $h(A)=0$ and then $\AE(f)=0$.


\section{Anqie-independence and semi-continuity of anqie entropy}\label{properties of anqie entropy}

Additivity for independent objects is one of the most crucial features of all kinds of entropies. To state this property for anqie entropy (see Theorem \ref{thm_entropy_additive}), we define anqie independence by tensor products as follows.

\begin{definition}
We call anqies $\cA_\xi$ $(\xi\in \Xi)$ \textbf{anqie independent} if the C*-algebra they generate, denoted by $\mathcal{A}$, is canonically isomorphic to  $\bigotimes_{\xi\in \Xi}\cA_\xi$ as a C*-algebra tensor product, or equivalently, the space $X$ is homeomorphic to the $\prod\nolimits_{\xi\in \Xi}X_\xi$, where $X,X_\xi$ are the maximal ideal spaces of $\cA,\cA_\xi$, respectively. Families of arithmetic functions $\cF_\xi$ $(\xi\in \Xi)$ are called \textbf{anqie independent} if $\cA_{\cF_\xi}$ $(\xi\in \Xi)$ are anqie independent. Here the symbolic $\Xi$ is an index set.
\end{definition}

\begin{remark}
{\rm
As far as we are concerned here, all C*-algebras involved are abelian. The tensor product of abelian C*-algebras has a unique C*-algebra tensor-product norm. Let $X_{1}$ and $X_{2}$ be the maximal
ideal spaces of $\A$ and $\B$, respectively. Then the C*-tensor-norm on $\mathcal{A}\otimes\mathcal{B}$ agrees with the norm on $C(X_{1}\times X_{2})$.
}\end{remark}

The following simple example shows how our anqie independence is related to certain arithmetic structures of
natural numbers.
\begin{example}\label{example of anqie independent}
{\rm Let $f_l=\{0,1,\ldots,l-1,0,1,\ldots,l-1,\ldots\}$, for any $l\geq 1$. Here $f_{l}$ is viewed as a periodic function of period $l$.
Then $f_{m}$ and $f_{n}$ are anqie independent if and only if $(m,n)=1$.
}\end{example}

In the following, we give detailed argument for the above claim. Let $\mathcal{A}$ be the anqie generated by $\{f_{m},f_{n}\}$ and $X$ the maximal ideal space of $\mathcal{A}$. Denote by $\mathcal{A}_{m}$ and $\mathcal{A}_{n}$ the anqies generated by $f_{m}$ and $f_{n}$, respectively. Then by Theorem \ref{thm_dynamicalGN}, the maximal ideal space of $\mathcal{A}_m$, denoted by $X_{1}$, is the set $f_{m}(\mathbb{N})$ consisting of $m$ elements. Similarly we use $X_{2}$ to denote that of $\mathcal{A}_{n}$, which is the set $f_{n}(\mathbb{N})$ consisting of $n$ elements. It is not hard to check that $X=\{(f_{m}(k),f_{n}(k)): k=0,1,\ldots\}$. Then $X\subseteq X_{1}\times X_{2}$. It is easy to see that $X=X_{1}\times X_{2}$ if and only if for any $i=0,\ldots,m-1$ and $j=0,\ldots,n-1$, the linear system of congruences $k\equiv i(mod~m)$, $k\equiv j(mod~n)$ has one solution. By Chinese remainder theorem, we conclude that $X=X_{1}\times X_{2}$ if and only if $(n,m)=1$. Hence the claim in the above example holds.

For $f_{\theta_i}(n)=e^{2\pi \i n\theta_i}$ with $\theta_i$ irrational $(0\leq i\leq t, t\in \mathbb{N})$, if $\theta_0,\ldots,\theta_t$ are $\bQ$-linearly independent, then $f_{\theta_{0}},\ldots,f_{\theta_{t}}$ are anqie independent. In the following, we give a brief explanation to the above fact. From Proposition \ref{uniformly distributed} and Lemma \ref{the polynomial is uniformly distributed}, we obtain that the sequence $\{(n\theta_0,\ldots, n\theta_t)\}_{n=0}^\infty$ is uniformly distributed modulo $1$. Then by Theorem \ref{thm_dynamicalGN}, the maximal ideal space of the anqie generated by $f_{\theta_{0}},\ldots, f_{\theta_{t}}$ is the Cartesian product of $t+1$ copies of $S^{1}$, while that of the anqie generated by $f_{\theta_{i}}$ is $S^{1}$, for $i=0,\ldots,t$.

Now we prove one of the main results in this section.
\begin{proof}[Proof of Theorem \ref{thm_entropy_additive}]
Set $\mathcal{F}_{1}=\{f_1,\ldots,f_n\}$, $\mathcal{F}_{2}=\{g_1,\ldots,g_m\}$ and $\mathcal{F}=\mathcal{F}_{1}\cup \mathcal{F}_{2}$.
Let $(X,A)$ (or, $\A$) be the anqie generated by $\mathcal{F}$. Denoted by $(X_{1},A_{1})$ (or, $\A_{1}$) and $(X_{2},A_{2})$ (or, $\A_{2}$) the anqies generated by $\mathcal{F}_{1}$ and $\mathcal{F}_{2}$, respectively. Then there is a continuous injective map, denoted by $i$, from $X$ to $X_{1}\times X_{2}$ given by $i(\rho)=(\rho|_{\mathcal{A}_{1}},\rho|_{\mathcal{A}_{2}})$ for any $\rho\in X$. In this case $(X,A)$ can be viewed as a subsystem of $\left(X_{1}\times X_{2}, A_{1}\times A_{2}\right)$, that is $X$ is a closed subset of $X_{1}\times X_{2}$ and the restriction of $A_{1}\times A_{2}$ to $X$ is identified with $A$. Note that $h(A_{1}\times A_{2})=h(A_{1})+h(A_{2})$ (see \cite[Theorem 3]{Adl-Kon-And65}). Then
\[
{\AE}(\mathcal{F})=h(A)\leq h(A_{1}\times A_{2})=h(A_{1})+h(A_{2})=\AE(\mathcal{F}_{1})+\AE(\mathcal{F}_{2}).
\]
If $\mathcal{F}_{1}$ and $\mathcal{F}_{2}$ are anqie independent, then $i$ is a homeomorphism and thus the equality holds.
\end{proof}

\begin{corollary} \label{cor_AE(f+g)_leq_AEf+AEg}
For any $f,g\in l^\infty(\bN)$, we have
\[
\left|{\AE}(f)-{\AE}(g)\right|\leq {\AE}(f\pm g)\leq {\AE}(f)+{\AE}(g),
\]
\[
{\AE}(f\cdot g)\leq {\AE}(f)+{\AE}(g).
\]
\end{corollary}

\begin{proof}
Since $\cA_{f \pm g}$ and $\cA_{f\cdot g}$ are subanqies of $\cA_{f,g}$, Lemma \ref{basic property}(\romannumeral1) and Theorem \ref{thm_entropy_additive} imply that
\[
{\AE}(f\pm g), {\AE}(f\cdot g)\leq {\AE}(f,g)\leq {\AE}(f)+{\AE}(g).
\]
The conclusions result from
${\AE}(f)={\AE}((f\pm g)\mp g)\leq {\AE}(f\pm g)+ {\AE}(g)$.
\end{proof}

As a simple application of the above corollary, we obtain the following result.

\begin{proposition}\label{a property of zero entropy} For any $f,g\in
l^\infty(\mathbb{N})$ with $\AE(g)=0$, we always have $\AE(f+g)=\AE(f)$.
\end{proposition}

\begin{proof}

From Corollary \ref{cor_AE(f+g)_leq_AEf+AEg}, we have $\AE(f+g)+\AE(-g)\ge \AE(f)$ and $\AE(f+g)\le
\AE(f)+\AE(g)$. It follows that $\AE(f+g)=\AE(f)$. \end{proof}


Write $\sigma_{A}(\cF)=\{\sigma_{A}f: f\in \cF\}$. Since $\mathcal{A}_{\sigma_{A}(\mathcal{F})}\subseteq \mathcal{A}_{\mathcal{F}}$, one has ${\AE}(\cA_{\sigma_{A}(\mathcal{F})})\leq {\AE}(\cA_\cF)$ by Lemma \ref{basic property}. From Theorem \ref{thm_dynamicalGN}, it is not hard to see that $X_\cF\setminus X_{\sigma_{A}(\mathcal{F})}$ contains at most one element. Thus the following corollary holds.

\begin{corollary} \label{cor_AE(AF)=AE(f)}
For any family $\cF\subseteq l^\infty(\bN)$, we have ${\AE}(\sigma_{A}(\cF))={\AE}(\cF)$.
\end{corollary}

The following lemma considers the anqie entropy of the inverse of an arithmetic function.

\begin{lemma}\label{quotient of zero entropy functions}
Let $c$ be a positive constant. Then for any $f\in l^{\infty}(\mathbb{N})$ with $|f(n)|>c$ for all $n\in \mathbb{N}$, one has $\AE(1/f)=\AE(f)$.
\end{lemma}

\begin{proof}
Let $(X_{f},B_{f})$ and $(X_{1/f},B_{1/f})$ be the canonical representations of $\mathcal{A}_{f}$ and $\mathcal{A}_{1/f}$, respectively. Applying Theorem \ref{thm_dynamicalGN},
the map $\varphi: X_{f}\rightarrow X_{1/f}$ defined by $\varphi((\omega_{0},\omega_{1},\ldots))=(1/\omega_{0},1/\omega_{1},\ldots)$, for $(\omega_{0},\omega_{1},\ldots)\in X_{f}$,
is a homeomorphism satisfying $\varphi\circ B_{f}=B_{1/f}\circ \varphi$. Then the topological entropy of $B_{f}$ equals that of $B_{1/f}$ and thus $\AE(f)=\AE(1/f)$.
\end{proof}

Now we prove Theorems \ref{infinity entropy is dense in the set of all functions} and \ref{thm_semi_continuous}.

\begin{proof} [Proof of Theorem \ref{infinity entropy is dense in the set of all functions}]
From Example \ref{an example of infinity entropy}, we see that $\mathcal{E}_{\infty}(\mathbb{N})\neq \emptyset$. Choose a $g\in \cE_{\infty}(\bN)$ such that $g\neq 0$ and $\AE(cg)=+\infty$ for any complex number $c\neq 0$. For each $f\in l^\infty(\bN)$ with $\AE(f)<+\infty$, we have
$\left\|\left(f+\epsilon\cdot g\right)-f\right\|_{l^\infty}= \epsilon\cdot \|g\|_{l^\infty}\rightarrow  0~~(\varepsilon \rightarrow 0^+)$.
It follows from Corollary \ref{cor_AE(f+g)_leq_AEf+AEg} that
$\AE(f+\epsilon\cdot g)\geq \AE(\epsilon\cdot g)-\AE(f)=+\infty$.
\end{proof}

\begin{proof} [Proof of Theorem \ref{thm_semi_continuous}]
Without loss of generality, we assume that $\sup_{i\in \bN}\|f_i\|_{l^\infty} \leq 1/2$. Then for any $\cG=\{g_{0},g_{1},\ldots\}\subseteq l^\infty(\bN)$ with $\sup\nolimits_{i\geq 0}\|g_i-f_i\|_{l^{\infty}}<\delta<1/2$, all $\overline{f_i(\bN)}$ and $\overline{g_i(\bN)}$ $(i\in \bN)$ lie in the closed unit disc $\bD$ on the complex plane. Write $\iota_n=
\left(f_i(n+j)\right)_{i,j\in \bN}$ and $\tau_n=
\left(g_i(n+j)\right)_{i,j\in \bN}$, viewed as matrices of infinite dimension with columns given by elements in $\prod\nolimits_{i=0}^\infty\overline{f_i(\bN)}$ and $\prod\nolimits_{i=0}^\infty\overline{g_i(\bN)}$, respectively. The shift map $B$ on $(\bD^\bN)^{\mathbb{N}}$ sends $\omega=(\omega_{i,j})_{i,j\in \bN}$ to $B(\omega)=(\omega_{i,j+1})_{i,j\in \bN}$. Recalling Theorem \ref{thm_dynamicalGN}, we consider the canonical representation $(X_\cF, B_\cF)$ and $(X_\cG, B_\cG)$ of $\cA_\cF$ and $\cA_\cG$, respectively, where $B_\cF=B|_{X_\cF}$ and $B_\cG=B|_{X_\mathcal{G}}$. There is a metric $d$ on $(\bD^\bN)^{\mathbb{N}}$ given by
\[
d(\omega,\omega^\prime)=\sum\limits_{j\in \bN} \sum\limits_{i\in \bN} 2^{-j-i-2} |\omega_{i,j}-\omega^\prime_{i,j}|, \quad \text{for any}~  \omega=\{\omega_{i,j}\}_{i,j\in \bN}~\text{and} ~\omega^\prime=\{\omega^\prime_{i,j}\}_{i,j\in \bN}.
\]
It is not hard to see that $d(\iota_n,\tau_n)<\delta$ for any $n\in\mathbb{N}$ whenever $\sup_{i\geq 0}\|f_i-g_i\|_{l^{\infty}}<\delta$. Define open sets $U_{n,m}=\left\{\omega\in X_\cF:\, d(\omega,\iota_n)<10^{-m}\right\}$ and $V_{n,m}=\left\{\omega\in X_\cG:\, d(\omega,\tau_n)<10^{-m}\right\}$. Then $\cU_m=\{U_{n,m}:n\in \bN\}$ and $\cV_m=\{V_{n,m}:n\in \bN\}$ $(m\in \bN)$ are refining sequences of open covers for $X_\cF$ and $X_\cG$, respectively. By \cite[Property 12]{Adl-Kon-And65}, we obtain
\[
h(B_\cF)=\lim\limits_{m\rightarrow \infty} h(B_\cF, \cU_m),\quad h(B_\cG)=\lim\limits_{m\rightarrow \infty} h(B_\cG, \cV_m).
\]
Moreover, the sequences on the right side of above equalities are non-decreasing. Now for any $\epsilon>0$, there is a sufficiently large $m^\prime$ such that
\[
h(B_\cF, \cU_{m^\prime-1})> h(B_\cF)-\epsilon.
\]
For any $s\geq 1$, let $\beta_{s}$ be a subcover of $\bigvee\nolimits_{0\leq j\leq s-1}B_\cG^{-j}(\cV_{m^\prime})$ of $X_\cG$ 
satisfying that the cardinality of $\beta_s$ is equal to $\cN(\bigvee\nolimits_{j=0}^{s-1}B_\cG^{-j}(\cV_{m^\prime}))$ (the minimal cardinality of all possible subcovers). Suppose that
\[
\beta_s=\left\{V_{n_0,m^\prime}\cap B_\cG^{-1}(V_{n_1,m^\prime})\cap \cdot\cdot\cdot \cap B_\cG^{-s+1}(V_{n_{s-1},m^\prime}):(n_{0},n_{1},\ldots,n_{s-1})\in \mathcal{S}_{m^\prime}\right\},
\]
where $\mathcal{S}_{m'}$ is some subset of the Cartesian product of $s$ copies of $\mathbb{N}$.

We next show that
\[
\left\{U_{n_0,m^\prime-1}\cap B_\cF^{-1}(U_{n_1,m^\prime-1})\cap \cdots \cap B_\cF^{-s+1}(U_{n_{s-1},m^\prime-1}):(n_{0},n_{1},\ldots,n_{s-1})\in \mathcal{S}_{m^\prime}\right\},
\]
denoted by $\alpha_s$,
is an open cover for $X_\cF$. Indeed, for any $\omega\in X_\cF$, there is an $n^\prime\in \bN$ such that $d(\omega,\iota_{n^\prime})<10^{-m^\prime-s+1}$. One may verify that
\[
d\left(B_\cF^j \omega,\iota_{n^\prime+j}\right)=d\left(B_\cF^j \omega, B_\cF^j \iota_{n^\prime}\right)\leq 2^j \cdot d(\omega,\iota_{n^\prime})<2^j\cdot 10^{-m^\prime-s+1}<10^{-m^\prime}
\]
for $0\leq j\leq s-1$. Assume further that
\[
\tau_{n^\prime}\in V_{n_0,m^\prime}\cap B_\cG^{-1}(V_{n_1,m^\prime})\cap \cdot\cdot\cdot \cap B_\cG^{-s+1}(V_{n_{s-1},m^\prime}),
\]
which is an open set in the cover $\beta_s$. For $0\leq j\leq s-1$, the condition $\tau_{n^\prime}\in B_\cG^{-j}(V_{n_j,m^\prime})$ implies $d\left(\tau_{n^\prime+j}, \tau_{n_j}\right)<10^{-m^\prime}$. It follows that
\begin{align*}
d\left(B_\cF^j \omega,\iota_{n_j}\right)&\leq d\left(B_\cF^j \omega,\iota_{n^\prime+j}\right)+d(\iota_{n^\prime+j},\tau_{n^\prime+j})
+d(\tau_{n^\prime+j},\tau_{n_j})+d(\tau_{n_j},\iota_{n_j})\\
&< 10^{-m^\prime}+\delta+10^{-m^\prime}+\delta<10^{-m^\prime+1},\quad (0\leq j\leq s-1),
\end{align*}
provided that $\delta<10^{-m^\prime}$ (this choice is independent of $s$). Therefore,
\[
\omega\in U_{n_0,m^\prime-1}\cap B_\cF^{-1}(U_{n_1,m^\prime-1})\cap \cdot\cdot\cdot \cap B_\cF^{-s+1}(U_{n_{s-1},m^\prime-1})
\]
for any $s\geq 1$. So $\alpha_s$ is an open cover of $X_\cF$ and a subcover of $\bigvee\nolimits_{0\leq j\leq s-1}B_\cF^{-j}(\cU_{m^\prime-1})$. Hence one has
\[
\cN(\bigvee\nolimits_{j=0}^{s-1}B_\cF^{-j}(\cU_{m^\prime-1}))\leq \cN(\bigvee\nolimits_{j=0}^{s-1}B_\cG^{-j}(\cV_{m^\prime})).
\]
Letting $s$ tend to infinity, we then have
\begin{align*}
&h(B_\cG, \cV_{m^\prime})=\lim\limits_{s\rightarrow \infty} \frac{\log \cN(\bigvee\nolimits_{j=0}^{s-1}B_\cG^{-j}(\cV_{m^\prime}))}{s}\\
&\geq \lim\limits_{s\rightarrow \infty} \frac{\log \cN(\bigvee\nolimits_{j=0}^{s-1}B_\cF^{-j}(\cU_{m^\prime-1}))}{s}=h(B_\cF,\cU_{m^\prime-1})>h(B_\cF)-\epsilon.
\end{align*}
It follows that $h(B_\cG)\geq h(B_\cG, \cV_{m^\prime})>h(B_\cF)-\epsilon$, i.e., ${\AE}(\cG)>{\AE}(\cF)-\epsilon$.
\end{proof}

\begin{remark}
{\rm
In the statement of Theorem \ref{thm_semi_continuous}, if $\cF=\{f_{0},f_{1},\ldots\}$ is a family of bounded arithmetic functions with $\AE(\mathcal{F})=+\infty$, by similar proof to above, for any $N>0$, there is a $\delta>0$, such that ${\AE}(\cG)>N$ whenever $\cG=\{g_0,g_{1},\ldots\}\subseteq l^\infty(\bN)$ satisfies $\sup_{i\geq 0}\|g_i-f_i\|_{l^\infty}<\delta$.}
\end{remark}

\begin{corollary} \label{cor_semi_continuous}
View ${\AE}: \, l^\infty(\bN) \rightarrow [0,+\infty]$ as a map defined on the set of all bounded arithmetic functions. It is lower semi-continuous. Equivalently, if $\{f_k\}_{k=0}^\infty\subseteq l^\infty(\bN)$ is a sequence of functions that converges to $f\in l^\infty(\bN)$, i.e., $\lim\limits_{k\rightarrow \infty}\|f_k-f\|_{l^{\infty}}=0$, then ${\AE}(f)\leq \liminf\limits_{k\rightarrow \infty} {\AE}(f_k)$.
\end{corollary}

\section{The arithmetic compactification of natural numbers}\label{the arithmetic compactification of natural numbers}
In this section we shall use properties of anqie entropy to construct a new compactification of natural numbers which is comparable to the Stone-${\rm\check{C}}$ech compactification.
The function space of our compactification is described by the following theorem.

\begin{theorem} \label{thm_C_0_subCalg}
Let $\cE_0(\bN)$ be the set of all functions in $l^\infty(\bN)$ with vanishing anqie entropy, i.e.,
$\cE_0(\bN)=\left\{f\in l^\infty(\bN):\, {\AE}(f)=0\right\}$.
Then $\cE_0(\bN)$ is a subanqie of $l^\infty(\bN)$.
\end{theorem}

\begin{proof}
Corollary \ref{cor_AE(f+g)_leq_AEf+AEg} shows that $\cE_0(\bN)$ is an algebra closed under complex conjugation.
Suppose that $f\in l^\infty(\bN)$ and $\{f_k\}_{k=0}^\infty$ is a sequence in $\cE_0(\bN)$
with $\lim\limits_{k\rightarrow\infty}\|f_k-f\|_{l^{\infty}}=0$. By Corollary \ref{cor_semi_continuous}, one has $0\leq {\AE}(f)\leq \liminf\limits_{k\rightarrow \infty} {\AE}(f_k)=0$.
This implies $f\in \cE_0(\bN)$. Hence, the algebra $\cE_0(\bN)$ is closed under $l^\infty$-norm.
By Corollary \ref{cor_AE(AF)=AE(f)}, it is also $\sigma_{A}$-invariant.
Thus $\cE_0(\bN)$ is an anqie.
\end{proof}

\begin{definition}
Let $\cE_0(\bN)$ be the set of all functions $f$ in $l^\infty(\bN)$ with $\AE(f)=0$. Denote by $E_0(\bN)$ the maximal ideal space of $\mathcal{E}_0(\bN)$,
and call it the \textbf{arithmetic compactification} of natural numbers.
\end{definition}

By definition, $\mathcal E_0(\mathbb{bN})\cong C(E_0(\mathbb{bN}))$. Define $\delta_m(n)$ to be $1$ when $n=m$, and $0$ otherwise. Clearly, for any $m\in \bN$ the function $\delta_m(n)$ has zero anqie entropy. Denote $C_{0}(\bN)$ by $\{f: \lim_{n\rightarrow \infty}f(n)=0\}$. Then $C_0(\bN)\subseteq \mathcal{E}_0(\bN)$. Each $\iota(m)\in \iota(\bN)$ is an isolated point in $E_0(\bN)$. In fact, the set $\{\omega\in E_{0}(\mathbb{N}): |\omega(\delta_{m})-\iota(m)(\delta_{m})|<\frac{1}{2}\}$ is a neighborhood of $\iota(m)$ in $E_{0}(\mathbb{N})$, which has only the point $\iota(m)$.

The following lemma is a generalization of \cite[Theorem 3]{Adl-Kon-And65}, where a finite product is replaced by an infinite one. The proof is similar. We omit the details here.

\begin{lemma} \label{lem_entropyleq_infinitecase}
Let $(X_\xi,T_\xi)$ $(\xi\in \Xi)$ be topological dynamical systems and $(X,T)$ be the product system $(\prod\nolimits_{\xi\in \Xi}X_\xi, \prod\nolimits_{\xi\in \Xi}T_\xi)$
indexed by $\Xi$. Then
\[
h\left(T\right)=\sum_{\xi\in \Xi}h(T_\xi).
\]
\end{lemma}

\begin{lemma}
With the notation given in Theorem \ref{thm_C_0_subCalg}, we have $\AE(\mathcal{E}_{0}(\mathbb{N}))=0.$
\end{lemma}

\begin{proof}
Choose $\mathcal{F}=\mathcal{E}_{0}(\mathbb{N})$ in Theorem \ref{thm_dynamicalGN}, then $(E_0(\bN),A)$ can be viewed as a subsystem of $\left(\prod_{f\in \mathcal{E}_0(\bN)}X_f,\prod_{f\in \mathcal{E}_0(\bN)}A_f\right)$, where $X_{f}$ is the maximal ideal space of $\mathcal{A}_{f}$ and $A_{f}$ is the continuous map on $X_{f}$ extended by the map $\iota (n)\mapsto \iota (n+1)$ on $\iota(\mathbb{N})$. By Lemma \ref{lem_entropyleq_infinitecase}, $\AE(\mathcal{E}_{0}(\mathbb{N}))=0$.
\end{proof}

Each topological system has at least a maximal zero entropy factor \cite[Theorem 6.2.7]{YHS}.
Indeed, for $(\beta \bN, A)$, where $\beta\bN$ is the Stone-${\rm\check{C}}$ech compactification of $\bN$,
the system $(E_0(\bN), A)$ is exactly the one.

\begin{proposition} \label{thm_E_0_universal_property}
The following universal properties hold.

(\romannumeral1) If $\cA$ is an anqie with ${\AE}(\cA)=0$, then $\cA\subseteq \cE_0(\bN)$.

(\romannumeral2) The topological dynamical system $(E_0(\mathbb{N}), A,\iota(0))$ satisfies the following property:
for any point transitive dynamical system $(X,T,x_0)$ with $h(T)=0$,
there is a factor map $\pi: (E_0(\mathbb{N}),A,\iota(0))\rightarrow (X,T,x_0)$ such that $\pi(\iota(0))=x_0$.
Moreover, if another point transitive
system $(Z,K,z_0)$ with $h(K)=0$ satisfies this property, then it is equivalent to $(E_0(N),A,\iota(0))$.
\end{proposition}

\begin{proof}
(\romannumeral1) For any $f\in \cA$, the anqie $\cA_f$ is a subanqie of $\cA$. It follows that ${\AE}(f)\leq {\AE}(\cA)=0$. Thus $\cA\subseteq \cE_0(\bN)$.

(\romannumeral2) Since $x_{0}$ is a transitive point for $(X,T)$, the map $n\mapsto T^{n}x_{0}$ gives rise to a map from $\mathbb{N}$ to $X$ with a dense range.
It induces an injective homomorphism, denoted by $\varrho$, from $C(X)$ into
$l^\infty(\mathbb{N})$, i.e., $(\varrho f)(n)=f(T^{n}x_{0})$ for any $f\in C(X)$ and $n\in\mathbb{N}$.
So $\varrho(C(X))$ is a C*-subalgebra of $l^{\infty}(\mathbb{N})$. Note that $f\circ T\in C(X)$ for any $f\in C(X)$.
Thus $\varrho(C(X))$ is $\sigma_{A}$-invariant and an anqie of $\mathbb{N}$. Suppose that $Y$ is the maximal ideal space of $\varrho(C(X))$.
It is easy to see that $\iota(n)\mapsto T^{n}x_{0}$ extends to a homeomorphism from $Y$ to $X$, also denoted by $\varrho$.
Hence $(Y,A,\iota(0))$ and $(X,T,x_{0})$ are equivalent. Since $\AE(\varrho(C(X)))=h(T)=0$, it follows from (\romannumeral1) that $\varrho(C(X))$
is a subanqie of $\mathcal{E}_{0}(\mathbb{N})$. Thus $(Y,A)$ is a topological factor of $(E_0(\mathbb{N}), A)$.
The factor map $\phi$ can be chosen as in equation (\ref{eq_induced_between_A_1_and_A_2}).
Therefore, the map $\pi=\varrho\circ\phi: (E_0(\mathbb{N}), A,\iota(0))\rightarrow (X,T,x_{0})$ is a factor map with $\pi(\iota(0))=x_0$.

Let $(Z,K,z_0)$ be another point transitive dynamical system with $h(K)=0$ satisfying the above properties.
Then there is a continuous surjective map $\varphi:E_0(\mathbb{N})\rightarrow Z$ such that $\varphi\circ A=K\circ\varphi$ and
$\varphi(\iota(0))=z_{0}$. It follows that $\varphi\circ A^{n}(\iota(0))=K^{n}\circ\varphi(\iota(0))$ and $\varphi(\iota(n))=K^{n}z_{0}$.
Symmetrically, there is a continuous surjective map $\psi:Z\rightarrow E_0(\mathbb{N})$ such that $\psi\circ K=A\circ\psi$
and $\psi(z_{0})=\iota(0)$. It follows that $\psi\circ K^{n}=A^{n}\circ\psi$ and $\psi(K^{n}z_{0})=\iota(n)$.
It is easy to see that $\psi$ is the inverse of $\varphi$. Thus $\varphi$ is a homeomorphism.
Therefore $(Z,K,z_0)$ and $(E_0(\mathbb{N}),A,\iota(0))$ are equivalent.
\end{proof}
To get a better understanding of functions with zero anqie entropy, we show some properties of the arithmetic compactification below.

\begin{proposition} \label{thm_ari_compa_property}
The arithmetic compactification $E_0(\mathbb{N})$ has the following properties.

(\romannumeral1) It is not extremely disconnected.

(\romannumeral2) Let $R$ be an infinite subset of $\bN$. The closure of the set $\iota (R)$ in $E_0(\bN)$ is uncountable.

(\romannumeral3) The space $E_0(\bN)$ is not metrizable.
\end{proposition}

\begin{proof}
(\romannumeral1) Assume on the contrary that $E_0(\bN)$ is extremely disconnected. Then the closure of any open set is still open.
It is not hard to construct  a characteristic function $\chi_{S_{1}}(n)$ defined on a certain subset $S_{1}$ of $\mathbb{N}$ with $\AE(\chi_{S_{1}})\neq 0$. For example, choose $(\chi_{S_{1}}(0),\chi_{S_{1}}(1),\ldots)\in \{0,1\}^{\mathbb{N}}$ containing every finite sequences of 0's and 1's. Then by Lemma \ref{calculation formula of entropy}, $\AE(\chi_{S_{1}})=\log 2$.
Let $S_{0}=\mathbb{N}\setminus S_{1}$. Since each point $\iota (n)$ $(n\in \bN)$ is isolated in $E_{0}(\mathbb{N})$, $\iota(S_{1})$ is an open set in $E_{0}(\mathbb{N})$. Let $\overline{\iota (S_{1})}$ be the closure of $\iota (S_{1})$ in $E_{0}(\mathbb{N})$. Then by the assumption, we have that $\overline{\iota (S_{1})}$ in $E_{0}(\mathbb{N})$ is clopen and thus $\chi_{\overline{\iota (S_{1})}}\in C(E_{0}(\mathbb{N}))$. Since $\overline{\iota (S_{1})}\cap \iota(S_{0})=\emptyset$, the preimage of $\chi_{\overline{\iota (S_{1})}}$ under the inverse of the Gelfand transform (see equation (\ref{Gelfand transform})) in $\mathcal{E}_{0
}(\mathbb{N})$ is $\chi_{S_{1}}$. This contradicts the fact $\AE(\chi_{S_{1}})\neq 0$.

(\romannumeral2) Write $R=\{n_k\}_{k=0}^\infty$ with $n_0<n_1<n_2<\ldots$. Note that the sequence $\{n_k\theta\}_{k=0}^\infty$ is uniformly distributed modulo $1$ for almost all $\theta\in (0,1)$ (see \cite[Theorem 11.2.5]{Mur}). Choose one such $\theta$, and consider the transitive dynamical system $(S^1,R_\theta,1)$ with zero topological entropy, where $R_\theta$ is the irrational rotation with the angle $\theta$ on $S^{1}$. By Proposition \ref{thm_E_0_universal_property}, there is a factor map $\varphi$ from $(E_0(\bN),A)$ onto $(S^1,R_\theta)$ such that $\varphi(\iota(0))=1$. Moreover, $\{\varphi(\iota(n_k))\}_{k=0}^\infty=\{e^{2\pi i n_k \theta}\}_{k=0}^\infty$. The closure of the latter is $S^1$, which is uncountable. Therefore, the closure of $\{\iota (n_k)\}_{k=0}^\infty$ in $E_0(\bN)$ is also uncountable.

(\romannumeral3) Assume on the contrary that $E_0(\bN)$ is metrizable. Choose $\omega\in E_{0}(\mathbb{N})\setminus \iota(\mathbb{N})$. Since $E_{0}(\mathbb{N})$ is a metric space by the assumption, there is a sequence of integers $\{\iota (n_{k})\}_{k=0}^{\infty}$ such that $\iota (n_{k})$ converges to $\omega$ in $E_{0}(\mathbb{N})$ as $k\rightarrow \infty$. This contradicts the fact in (\romannumeral2) that the set $\{\iota (n_{0}),\iota (n_{1}),\ldots\}$ has infinite limit points.
\end{proof}

In general, for a compact Hausdorff space $X$, if $X$ is extremely disconnected, then it is totally disconnected.
However the converse may not be true. We have proved that $E_{0}(\mathbb{N})$ is not extremely disconnected. So it is interesting to
ask if $E_0(\bN)$ is totally disconnected.
We shall prove that $E_0(\bN)$ is totally disconnected in the next section.
At the end of this section, we show an equivalent condition about the statement that $E_{0}(\mathbb{N})$ is totally disconnected.

\begin{proposition}\label{the equivalent condition of totally disconnected}
Suppose that $\mathcal{F}_{0}(\mathbb{N})$ is the set of all arithmetic functions with zero anqie entropy and finite ranges. Then $E_{0}(\mathbb{N})$ is totally disconnected
if and only if $\mathcal{E}_{0}(\mathbb{N})=\overline{\mathcal{F}_{0}(\mathbb{N})}$, where $\overline{\mathcal{F}_{0}(\mathbb{N})}$ is
the closure of $\mathcal{F}_{0}(\mathbb{N})$ in $l^{\infty}(\mathbb{N})$.
\end{proposition}

\begin{proof}
Assume that $E_{0}(\mathbb{N})$ is totally disconnected. Then for any $x,y\in E_{0}(\mathbb{N})$ with $x\neq y$,
there is a clopen set $U$ such that $x\in U, y\notin U$. Then $\chi_{U}\in C(E_{0}(\mathbb{N}))$ and $\chi_{U}(x)=1, \chi_{U}(y)=0$.
This implies that the *-subalgebra $\gamma(\mathcal{F}_{0}(\mathbb{N}))$ in $C(E_{0}(\mathbb{N}))$ separates the points in $E_{0}(\mathbb{N})$, where $\gamma$ is the Gelfand transform (see equation (\ref{Gelfand transform})).
Then by the Stone-Weierstrass theorem, we have $\gamma(\mathcal{F}_{0}(\mathbb{N}))$ is dense in $C(E_{0}(\mathbb{N}))$, and thus $\mathcal{E}_{0}(\mathbb{N})=\overline{\mathcal{F}_{0}(\mathbb{N})}$.

Conversely, suppose that $\mathcal{E}_{0}(\mathbb{N})=\overline{\mathcal{F}_{0}(\mathbb{N})}$. Given distinct points $x,y\in E_{0}(\mathbb{N})$,
by Urysohn's lemma, there is an $f\in C(E_{0}(\mathbb{N}))$ such that $f(x)\neq f(y)$. Then by the assumption,
there is a $g\in \mathcal{F}_{0}(\mathbb{N})$ such that $\widetilde{g}=\gamma(g)$ in $C(E_{0}(\mathbb{N}))$ satisfying $\widetilde{g}(x)\neq \widetilde{g}(y)$. Note that the continuous function $\widetilde{g}$ has a finite range. It is not hard to check that $\widetilde{g}^{-1}(\{\widetilde{g}(x)\})$ is a clopen
set in $E_{0}(\mathbb{N})$. Thus $x,y$ can be separated by a clopen set. It follows that $E_{0}(\mathbb{N})$ is totally disconnected.
\end{proof}


\section{Approximation method for arithmetic functions}\label{Approximation method for arithmetic functions}

In this section, we shall develop an approximation method for any map $f$ from $\mathbb{N}$ to $X$, where $X$ is a compact Hausdorff space. Recall that in Section 2, the anqie entropy of $f$ is defined to be the topological entropy of the Bernoulli shift $B$ (on $X^{\mathbb{N}}$) restricted to the space $X_{f}$, where $X_{f}$ is the closure of the set $\{(f(n),f(n+1),\ldots):n\in \mathbb{N}\}$ in $X^{\mathbb{N}}$.

In the approximation process, we need to compute anqie entropies of maps with finite ranges. First let us recall some basic concepts in symbolic dynamical systems. For a finite set $\mathbb{A}$,
a \emph{block} over $\mathbb{A}$ is a finite sequence of symbols from $\mathbb{A}$.  An \emph{$m$-block}
is a block of length $m$. For any given (finite or infinite) sequence $x=(x_{0},x_{1},\ldots)$ of symbols
from $\mathbb{A}$, we say that a block $w$ \emph{occurs in} $x$ or $x$ \emph{contains} $w$ if there are
natural numbers $i$, $j$ with $i\leq j$, such that $(x_{i},\ldots,x_{j})=w$.  A \emph{concatenation} of
two blocks $w_{1}=(a_{1},\ldots,a_{k})$ and $w_{2}=(b_{1},\ldots,b_{l})$ over $\mathbb{A}$ is the block
$w_{1}w_{2}=(a_{1},\ldots,a_{k},b_{1},\ldots,b_{l})$.

Now suppose that $f:\mathbb{N}\rightarrow X$ has a finite range. We can view $f$ as a sequence
$\{f(n)\}_{n=0}^{\infty}$ in $f(\mathbb{N})^{\mathbb{N}}$. Let $B_{m}(f)$ denote the set of all $m$-blocks occurring in $f$, i.e.,
$
B_{m}(f)=\{(f(n),f(n+1),\ldots,f(n+m-1)):n\geq 0\}.
$

\begin{lemma}\label{calculation formula of entropy}
Let $X$ be a compact Hausdorff space and $f$ a map from $\mathbb{N}$ to $X$ with finite range. Then the~anqie~entropy~of~$f$ equals $\lim_{m\rightarrow \infty}\frac{\log|B_{m}(f)|}{m}$.
\end{lemma}

\begin{proof}
Let $X_{f}$ be the closure of the set $\{(f(n),f(n+1),\ldots,): n\in \mathbb{N}\}$ in $f(\mathbb{N})^{\mathbb{N}}$. Then by the definition of anqie entropy, we have that the anqie entropy of $f$ equals the topological entropy of $B_{f}$, the Bernoulli shift $B$ restricted to $X_{f}$. We assume that $f(\bN)=\{a_{1},\ldots,a_{k}\}$ for $k\geq 1$. Let $D_i=\{(x_0,x_1,\ldots)\in X_f: x_0=a_{i}\}$ for $1\leq i\leq k$. Denote by $\xi=\{D_1,\ldots,D_{k}\}$, an open cover of $X_{f}$. Define $\xi_{n}=\xi \vee  B_{f}^{-1} \xi \vee \cdot\cdot\cdot \vee B_{f}^{-n+1}\xi$.
Then $\eta=\{\xi_{n} : n\geq1\}$ is a refinement cover family of $X_{f}$. Thus $h(B_{f})=\lim_{n\rightarrow \infty}h(B_{f},\xi_n)$ (see \cite[Property 12]{Adl-Kon-And65}). Here the notation $h(B_{f},\xi_n)$ is introduced in Definition \ref{topological entropy}.
Let $\Phi_m(f)$ be the image of the projection map from $X_{f}$ onto its first $m$ coordinates. It is not hard to check that $h(B_{f},\xi_n)=\lim_{m\rightarrow \infty} \frac{\log
|\Phi_{m+n-1}(f)|}m=\lim_{m\rightarrow \infty} \frac{\log
|\Phi_{m}(f)|}m$. Note that the convergence of a sequence in $X_{f}$ is coordinate-wise. Then  $|\Phi_m(f)|=|B_{m}(f)|$ and the claim in this lemma holds.
\end{proof}

\begin{lemma}\label{regular blocks}
Let $X$ be a compact Hausdorff space and $f$ a map from $\mathbb{N}$ to $X$ with finite range. Denoted by $R_{m}(f)=
\{f(lm),f(lm+1),\ldots, f(lm+m-1):l\geq 0\}
$. Then the anqie entropy of $f$ equals
\begin{equation}\label{use regular block to compute the entreopy}
\lim_{m\rightarrow \infty}\frac{\log |R_{m}(f)|}{m}.
\end{equation}
\end{lemma}
\begin{proof}
On one hand, it follows from $|R_{m}(f)|\leq |B_{m}(f)|$ that $\limsup_{m\rightarrow \infty}\frac{\log |R_{m}(f)|}{m}$ $\leq \AE(f)$.
On the other hand, for any given $m\geq 1$ and any $km$-block $w$ with $k\geq 1$ occurring in $f$,
there is a concatenation of certain $k+1$ successive $m$-blocks in $R_{m}(f)$ such that the concatenation
contains $w$. Thus $|B_{km}(f)|\leq m(|R_{m}(f)|)^{k+1}$, which implies
$\AE(f)=\lim_{k\rightarrow \infty}\frac{\log |B_{km}(f)|}{km}\leq \frac{\log |R_{m}(f)|}{m}.$
We then have $\AE(f)\leq \liminf_{m\rightarrow \infty}\frac{\log |R_{m}(f)|}{m}$. Hence $\AE(f)=\lim_{m\rightarrow \infty}\frac{\log |R_{m}(f)|}{m}$.
\end{proof}

Now we prove a stronger version of Theorem \ref{theorem_finite_appoximate}.
\begin{theorem} \label{implification theorem}
Suppose that $(X,d)$ is a compact metric space and $f$ a map from $\mathbb{N}$ to $X$ with anqie entropy $\lambda$ $(0\leq \lambda<+\infty)$. Then for any $N\geq 1$, there is a map $f_{N}$ from $\mathbb{N}$ to $f(\mathbb{N})$ with finite range such that the anqie entropy of $f_{N}$ is less than or equal to $\lambda$ and $\sup_{n}d(f_{N}(n),f(n))\leq \frac{1}{N}$.
\end{theorem}

\begin{proof}
Given $N\geq 1$, suppose that $\{U_{i}=B_{d}(x_{i},\frac{1}{N}):i=0,1,\ldots,k\}$ is an open cover of $\overline{f(\mathbb{N})}$,
where $B_{d}(x_{i},\frac{1}{N})=\{x\in X: d(x_{i},x)<\frac{1}{N}\}$. We may assume that $x_{i}=f(m_{i})$.
For simplicity, we use $X_{0}$ to denote $\overline{f(\mathbb{N})}$.
Suppose that $X_{f}$ is the closure of the set $\{(f(n),f(n+1),\ldots): n\in \mathbb{N}\}$ in $X_{0}^{\mathbb{N}}$.
Let $B_{f}$ be the Bernoulli shift on $X_{0}^{\mathbb{N}}$ restricted to $X_{f}$, given by $(\omega_{0},\omega_{1},\ldots)\mapsto (\omega_{1},\omega_{2},\ldots)$.
For $s\geq 1$, denote $\cW_s$ by
\[
\left\{U_{i_0}\times U_{i_1}\times \ldots \times U_{i_{s-1}}\times X_0^{\bN\setminus\{0,1,\ldots,s-1\}}:\, \{i_0,i_1,\ldots,i_{s-1}\}\in \{0,1,\ldots,k\}^{s}\right\},
\]
which is an open cover for $X_f$. Set $t_0=1$ and $\cU^{(0)}=\cW_1$.

We use iteration on $l$ for $l=0,1,2,\ldots$. At the beginning of the $l$-th step, we always assume that there is a natural number $t_l$ and an open cover $\cU^{(l)}$ of $X_f$ which is a subcover of $\cW_{t_l}$. From the definition of the anqie entropy of $f$, we have that the topological entropy of $B_{f}$, denoted by $h(B_{f})$, is equal to $\lambda$. So $h(B_{f}^{t_l})=t_l\lambda$. For each $s\geq 1$, write $\cU_s^{(l)}= \bigvee\nolimits_{j=0}^{s-1}(B_{f}^{t_l})^{-j}\mathcal{U}^{(l)}$. Then $\lim\limits_{s\rightarrow \infty}s^{-1}\log\cN(\cU_s^{(l)})\leq t_l\lambda$. Here, for an open cover $\mathcal{U}$ of $X_{f}$, recall that $\mathcal{N}(\mathcal{U})$ denotes the minimal number of open sets in $\mathcal{U}$ that cover $X_{f}$. So there is a sufficiently large natural number $s_l$ 
such that $s_{l}^{-1}\log\cN(\cU_{s_{l}}^{(l)})< t_l\lambda+2^{-l}$. Set $t_{l+1}=t_ls_l$. Then $\cU_{s_l}^{(l)}$ is a subcover of $\cW_{t_{l+1}}$. Now choose a subcover $\mathcal{V}^{(l)}=\{V_{1}^{(l)},V_{2}^{(l)},\ldots,V_{k_l}^{(l)}\}$ of $\cU_{s_l}^{(l)}$ which satisfies $|\cV^{(l)}|=\cN(\cU_{s_l}^{(l)})=k_{l}$.

Note that a point $(f(j),f(j+1),\ldots)$ $(j\in \bN)$ may lie in many open sets in $\cV^{(l)}$. We employ the following strategy to chose a particular one. Let $\pi_l:\, \{0,1,\ldots,k\}^{t_{l+1}} \rightarrow \cW_{t_{l+1}}$ be the bijection defined by
\[
\pi_l\left((i_{0},i_{1},\ldots,i_{t_{l+1}-1})\right)=U_{i_{0}}\times U_{i_{1}}\times \cdot\cdot\cdot\times U_{i_{t_{l+1}-1}}\times X_{0}^{\mathbb{N}\setminus \{0,1,\ldots, t_{l+1}-1\}}.
\]
Define
\[
i_{j}=\min\{1\leq i\leq k_l:\, \left(f(jt_{l+1}),\ldots,f((j+1)t_{l+1}-1), f((j+1)t_{l+1}),\ldots\right)\in V_i^{(l)}\}
\]
and $g_l$ to be the map from $\mathbb{N}$ to $\{0,1,\ldots,k\}$ by
\[
(g_l(jt_{l+1}),\ldots, g_l((j+1)t_{l+1}-1))=\pi_l^{-1}(V_{i_j}^{(l)})
\]
for each $j\geq 0$. Then we set $\cU^{(l+1)}= \cV^{(l)}$ and iterate on $l+1$.

Now, by the above construction $\pi_{l+1}^{-1}(\cV^{(l+1)})\subseteq \left(\pi_l^{-1}(\cV^{(l)})\right)^{s_{l+1}}$ for all $l\geq 0$. It follows that
for any $r\geq 0$, the sequence $(g_{r}(0),g_{r
}(1),\ldots)$ can be viewed as an infinite concatenation of at most $e^{(t_l\lambda+2^{-l})s_l}$ different $t_{l+1}$-blocks for each $l\leq r$.

Next we construct a map $g:\, \bN \rightarrow \{0,1,2,\ldots,k\}$ by $g(m)=g_0(m)$ for $0\leq m<t_1$ and $g(m)=g_l(m)$ for $t_l\leq m<t_{l+1}$ 
$(l\geq 1)$. It is not hard to verify that
\[
|\left\{\left(g(jt_{l+1}),\ldots, g((j+1)t_{l+1}-1)\right):\, j\geq 0\right\}| < e^{(t_l\lambda+2^{-l})s_l}+1
\]
for all $l\geq 0$. Here the term $+1$ counts the possibility that $(g(0),\ldots,g(t_{l+1}-1))$ does not belong to $\pi_l^{-1}(\cV^{(l)})$. It follows from Lemma \ref{regular blocks} that
\[
\AE(g)= \lim\limits_{l\rightarrow \infty}\frac{1}{t_{l+1}}\log |\left\{\left(g(jt_{l+1}),\ldots, g((j+1)t_{l+1}-1)\right):\, j\geq 0\right\}| \leq \lambda.
\]

Finally, define $f_{N}(n)=x_{g(n)}=f(m_{g(n)})$. From the above construction of $g$, we see $f(n)\in U_{g(n)}$. Note that $U_{g(n)}=B_{d}(x_{g(n)},\frac{1}{N})$. Then $d(f(n),f_{N}(n))=d(f(n),x_{g(n)})<\frac{1}{N}$. By Lemma \ref{calculation formula of entropy}, we conclude that the anqie entropy of $f_{N}$ is equal to $\AE(g)\leq \lambda$.
\end{proof}

Theorem \ref{theorem_finite_appoximate} is a consequence of the above theorem. Now we are ready to prove Theorem \ref{arithmetic compactification is totally disconnected}.

\begin{proof}[Proof of Theorem \ref{arithmetic compactification is totally disconnected}]
Applying Theorem \ref{implification theorem} and Proposition \ref{the equivalent condition of totally disconnected},
we obtain that the space $E_{0}(\mathbb{N})$ is totally disconnected. From Proposition \ref{thm_ari_compa_property}(\romannumeral1), we know that $E_{0}(\mathbb{N})$ is not extremely disconnected. The proof is completed.
\end{proof}

As an application of Theorem \ref{implification theorem}, we now prove Proposition \ref{approximation result of orbit}.
\begin{proof}[Proof of Proposition \ref{approximation result of orbit}]
Let $x$ be a given point in $X$ and $g$ be the map from $\mathbb{N}$ to $X$ defined by $g(n)=T^{n}x$.  Let $\mathcal{O}_x$ denote the orbit of $x$, i.e., $\mathcal{O}_x=\{T^{n}x: n=0,1,2,\ldots\}$. We use $\overline{\mathcal{O}}_x$ to
denote the closure of $\mathcal{O}_x$ in $X$. Then the topological entropy of $T$ restricted to $\overline{\mathcal{O}}_x$ is less than or equal to $\lambda$. By the definition of anqie entropy, we have that the anqie entropy of $g$ equals the topological entropy of $T$ restricted on $\overline{\mathcal{O}}_x$, which is less than or equal to $\lambda$. Then by Theorem \ref{implification theorem}, the claim in the proposition holds.
\end{proof}

At the end of this section, we prove the following result, which is crucial to prove Theorems \ref{thm_K0_group} and \ref{thm_K1_group} in the next section. The proof of this result is similar to that in Theorem \ref{implification theorem}.

\begin{lemma}\label{subset with zero entropy}
Suppose that $X$ is a compact Hausdorff space and $f$ a map from $\mathbb{N}$ to $X$ with anqie entropy
$\lambda$ $(0\leq \lambda<+\infty)$. Let $U_{1}$ be an open set in $X$, and $K\subseteq U_{1}$ be a closed
set in $X$. Then there is a set $C\subseteq \mathbb{N}$ with $f^{-1}(K)\subseteq C\subseteq f^{-1}(U_{1})$ such that
$\AE(\chi_{C})\leq \lambda$.
\end{lemma}

\begin{proof}
Let $U_{0}=X\setminus K$, then $X= U_{0}\cup U_{1}$. Suppose that $X_{f}$ is the closure of the set $\{(f(n),f(n+1),\ldots): n\in \mathbb{N}\}$ in $X^{\mathbb{N}}$. For $s\geq 1$, let
\[\mathcal{W}_{s}=\{U_{i_{0}}\times U_{i_{1}}\times\cdot\cdot\cdot\times U_{i_{s-1}}\times X^{\mathbb{N}\setminus \{i_{0},i_{1},\ldots,i_{s-1}\}}:\{i_{0},i_{1},\ldots, i_{s-1}\}\in \{0,1\}^{s}\},\]
which is an open cover for $X_{f}$. Then similar to the proof of Theorem \ref{implification theorem}, we can construct an arithmetic function $g: \mathbb{N}\rightarrow \{0,1\}^{\mathbb{N}}$ such that $\AE(g)\leq \lambda$ and $f(n)\in U_{g(n)}$. Moreover, for $n\in f^{-1}(K)$, we have $f(n)\in U_{1}\setminus U_{0}$ and then $g(n)=1$. For $n\notin f^{-1}(U_{1})$, we have $f(n)\in U_{0}\setminus U_{1}$, then $g(n)=0$. The lemma follows by taking $C=\{n: g(n)=1\}$.
\end{proof}

\section{The $K$-groups of $\mathcal{E}_{0}(\mathbb{N})$}\label{kgroups}

Recall that $\mathcal{E}_{0}(\mathbb{N})$ is the C*-algebra of all arithmetic functions with zero anqie entropy.
It is *-isomorphic to $C(E_{0}(\mathbb{N}))$. In this section, we shall prove that $K_{0}(\mathcal{E}_{0}(\mathbb{N}))\cong \{f\in \mathcal{E}_{0}(\mathbb{N}): f(\mathbb{N})\subseteq \mathbb{Z}\}$ (Theorem \ref{thm_K0_group}) and $K_{1}(\mathcal{E}_{0}(\mathbb{N}))=0$ (Theorem \ref{thm_K1_group}).
Lemma \ref{subset with zero entropy} is an essential tool to prove the above results. The following proposition follows from Lemma \ref{calculation formula of entropy} that will be used in our proof.

\begin{proposition}\label{01 valued function}
Let $f\in \mathcal{F}_{0}(\mathbb{N})$, i.e., $f$ has zero anqie entropy and finite range. For any $c\in f(\mathbb{N})$, let $f_{c}\in l^{\infty}(\mathbb{N})$ be the characteristic function defined on $f^{-1}(\{c\})$. Then $\AE(f_{c})=0$. In particular, any function in $\mathcal{F}_{0}(\mathbb{N})$ is a linear combination of $\{0,1\}$-valued functions with zero anqie entropy.
\end{proposition}

We first list some notation and recall the definition of $K_{0}$-group of unital C*-algebras. Let $\mathcal{A}$ be a unital C*-algebra. Denote $M_{k,l}(\mathcal{A})$ as the set of all $k\times l$ matrices with entries in $\mathcal{A}$.
In particular, $M_{k,k}(\mathcal{A})$, also denoted by $M_{k}(\mathcal{A})$, is a C*-algebra.
For $a_1,\ldots,a_k\in \cA$, the expression $\diag(a_1,a_2,\ldots,a_k)$ stands for the $k\times k$ diagonal matrix with diagonal elements $a_1,a_2,\ldots,a_k$ in order.
Denote by $\mathcal{P}(\mathcal{A})$ the set of all projections in $\mathcal{A}$, i.e.,
\[
\mathcal{P}(\mathcal{A})=\{p\in \mathcal{A}: p^2=p^{*}=p\}.
\]
Set
\[
\mathcal{P}_{k}(\mathcal{A})=\mathcal{P}(M_{k}(\mathcal{A}))~{\rm and}~\mathcal{P}_{\infty}(\mathcal{A})=\cup_{k=1}^{\infty}\mathcal{P}_{k}(\mathcal{A}).
\]
Here we view $\mathcal{P}_{k}(\mathcal{A})$, $k=1,2,\ldots$, as being pairwise disjoint.
For any $p\in \mathcal{P}_{k}(\mathcal{A})$ and $q\in \mathcal{P}_{l}(\mathcal{A})$, we say $p\sim_{0}q$ if and only if there is a $u\in M_{l,k}(\mathcal{A})$, such that
$
p=u^{*}u, q=uu^{*}.
$
It is known that $\sim_{0}$ is an equivalence relation on $\mathcal{P}_{\infty}(\mathcal{A})$. Define
$\mathcal{D}(\mathcal{A})=\mathcal{P}_{\infty}(\mathcal{A})/\sim_{0}$, and let $[p]_{\mathcal{D}}\in \mathcal{D}(\mathcal{A})$ be the equivalence class of $p\in \mathcal{P}_{\infty}(\mathcal{A})$. The formula
\[
[p]_{\mathcal{D}}+[q]_{\mathcal{D}}:=[\diag(p,q)]_{\mathcal{D}}
\]
gives a binary operation on $\mathcal{D}(\mathcal{A})$ such that $(\mathcal{D}(\mathcal{A}),+)$ forms an abelian semigroup. Recall that the $K_0$-group $K_{0}(\mathcal{A})$ is defined to be the Grothendieck group of the semigroup $\mathcal{D}(\mathcal{A})$. More specifically, define an equivalence relation $\sim$ on $\cD(\cA)\times \cD(\cA)$ by $([p_{1}]_{\mathcal{D}}, [q_{1}]_{\mathcal{D}})\sim ([p_{2}]_{\mathcal{D}},[q_{2}]_{\mathcal{D}})$ if there is some $[p]_{\mathcal{D}}\in \mathcal{D}(\mathcal{A})$ such that
\[
[p_{1}]_{\mathcal{D}}+[q_{2}]_{\mathcal{D}}+[p]_{\mathcal{D}}=[q_{1}]_{\mathcal{D}}+[p_{2}]_{\mathcal{D}}+[p]_{\mathcal{D}}.
\]
Then $K_{0}(\mathcal{A})=(\mathcal{D}(\mathcal{A})\times \mathcal{D}(\mathcal{A}))/\sim$.

Now, let us return to the computation of $K_{0}(\mathcal{E}_{0}(\mathbb{N}))$. Recall that $\mathcal{E}_{0}(\mathbb{N})\cong C(E_{0}(\mathbb{N}))$ and $E_{0}(\mathbb{N})$ is totally disconnected by Theorem \ref{arithmetic compactification is totally disconnected}. We remark that Exercise 3.4 in \cite{RLL} gives a general result for the $K_{0}$-group of $C(X)$, where $X$ is a totally disconnected space. Here for the specific object $E_{0}(\mathbb{N})$, we present a different method to compute the $K_{0}$ group, from which we obtain a result (Lemma \ref{diagonalization}) that can not be deduced from the general $K$-theory.

For simplicity, we use $\mathcal{E}_{0}$ to denote $\mathcal{E}_{0}(\mathbb{N})$ in this section. We first prove that each projection in $\mathcal{P}_{k}(\mathcal{E}_{0})$ can be ``diagonalized".
\begin{lemma}\label{diagonalization}
For any $k\geq 1$ and any $P\in \mathcal{P}_{k}(\mathcal{E}_{0})$, there is a diagonal matrix $Q\in \mathcal{P}_{k}(\mathcal{E}_{0})$, such that $P\sim_{0}Q$.
\end{lemma}

\begin{proof}
For any $k\geq 1$ and $r$ with $0\leq r\leq k$, define
$$\ba{lll}
\mathcal{R}_{k}(r)=\{(p_{ij})_{1\leq i,j\leq k}\in \mathcal{P}_{k}(\mathcal{E}_{0}): \sum_{i=1}^{k}p_{ii}(n)=r~{\rm for~any}~n\in \mathbb{N}\}.
\ea$$

\textbf{Claim:} For any $P\in \mathcal{R}_{k}(r)$, there is a diagonal matrix $Q\in\mathcal{R}_{k}(r)$ such that $P\sim_{0} Q$.

First we show how to prove this lemma if the claim holds. Let $P=(p_{ij})_{1\leq i,j\leq k}\in \mathcal{P}_{k}(\mathcal{E}_{0})$. Define
$$
r(n)=p_{11}(n)+p_{22}(n)+\cdot\cdot\cdot+p_{kk}(n)={\rm rank}(P(n)),
$$
where $P(n)=(p_{ij}(n))_{1\leq i,j\leq k}$. Then by Corollary \ref{cor_AE(f+g)_leq_AEf+AEg} we have $r(n)\in \mathcal{E}_{0}$. Note that $r(n)$ has a finite range.
For any $i=0,1,\ldots,k$, set $R_{i}=\{n\in \mathbb{N}: r(n)=i\}$. Then $\AE(\chi_{R_{i}})=0$ by Proposition \ref{01 valued function}.
Let $E_{i}=\diag(1_{i},0_{k-i})\in M_{k}(\mathbb{C})$. Consider the projections
$$
\widetilde{P_{i}}=\chi_{R_{i}}P+(1-\chi_{R_{i}})E_{i},~i=0,1,\ldots,k.
$$
Then $\widetilde{P_{i}}\in \mathcal{R}_{k}(i)$. From the assumption of the claim, there is a $V_{i}\in M_{k,k}(\mathcal{E}_{0})$,
such that $V_{i}V_{i}^{*}=\widetilde{P_{i}}$ and $V_{i}^{*}V_{i}$ is a diagonal matrix in $\mathcal{R}_{k}(i)$.
Note that $P=\sum_{i=0}^{k}\chi_{R_{i}}P=\sum_{i=0}^{k}\chi_{R_{i}}\widetilde{P_{i}}$. Choose $U=\sum_{i=0}^{k}\chi_{R_{i}}V_{i}$.
Then $UU^{*}=P$ and $U^{*}U$ is a diagonal matrix in $\mathcal{P}_{k}(\mathcal{E}_{0})$.

Next, we prove the correctness of the claim. We use induction on $k$ with $k=1,2,\ldots$.
For $k=1$, the proof is trivial. Assume inductively that the claim holds for some $k-1$ and any $r$ with $0\leq r\leq k-1$, where $k\geq 2$. In the following, we show that the claim holds for $\mathcal{R}_{k}(r)$ with $0\leq r\leq k$.

When $r=0$, the proof is trivial. In the following, we assume that $r\geq 1$. Suppose that $P=(p_{ij})_{1\leq i,j\leq k}\in \mathcal{R}_{k}(r)$.
Since $P$ is a projection, it follows from $P^*P=P$ that $0\leq p_{ii}\leq 1~(1\leq i\leq k)$. Let $p_{1}=p_{11}$. By Lemma \ref{subset with zero entropy},
there is a $C_{1}\subseteq \mathbb{N}$ with $\AE(\chi_{C_{1}})=0$,
such that
$$
p_{11}^{-1}([{r}/{k},1])\subseteq C_{1}\subseteq p_{11}^{-1}(({r}/{2k},1]).
$$
In the following, we use iteration on $l$ to construct a function $p_l$ and a set $C_l$ $(2\leq l\leq k)$. Let $p_l=p_{ll}\cdot\prod\nolimits_{j=1}^{l-1} (1-\chi_{C_j})$. Equivalently,
\[
p_l(n) =
\begin{cases}
0,\quad &\text{if }n\in C_1\cup \ldots \cup C_{l-1},\\
p_{ll}(n),\quad &\text{if }n\notin C_1\cup \ldots \cup C_{l-1}.
\end{cases}
\]
Note that $\AE(p_l)=0$. Applying Lemma \ref{subset with zero entropy} again, one obtains a set $C_l\subseteq \mathbb{N}$ with $\AE(\chi_{C_l})=0$ such that
\[
p_l^{-1}([r/k,1])\subseteq C_l\subseteq p_l^{-1}((r/2k,1]).
\]
Since $p_l(n)>r/2k>0$ for $n\in C_l$ $(2\leq l\leq k)$, one sees that the sets $C_1,\ldots,C_k$ are pairwise disjoint. Moreover, we have $\cup_{i=1}^{k}C_{i}=\mathbb{N}$. Actually, if $\cup_{i=1}^{k}C_{i}\neq \mathbb{N}$,
then choose $n\in \mathbb{N}\setminus \cup_{i=1}^{k}C_{i}$. Since $p_{i}(n)=p_{ii}(n)<\frac{r}{k}$ $(1\leq i\leq k)$,
this contradicts the fact that $\sum_{i=1}^{k}p_{ii}(n)=r$. Thus $\cup_{i=1}^{k}C_{i}=\mathbb{N}$ and $P=\sum_{i=1}^{k}\chi_{C_{i}}P$.

As matrices in $M_{k}(\mathcal{E}_{0})$, let $F_{1}=\cdot\cdot\cdot=F_{r}=$ diag$(1_{r},0_{k-r})$ and $F_{i}=$ diag$(0_{i-r},1_{r},$ $0_{k-i})$ for $r+1\leq i\leq k$.
Let $P_{i}=\chi_{C_{i}}P+(1-\chi_{{C}_{i}})F_{i}$ for $1\leq i\leq k$. Then $P=\sum_{i=1}^{k}\chi_{C_{i}}P_{i}$.
Suppose that $P_{i}=(f_{hl}^{i})_{1\leq h,l\leq k}\in \mathcal{R}_{k}(r)$. It is not hard to check that $f_{ii}^{i}(n)>{r}/{2k}$ for all $n$ by the construction of $C_{i}$.
Let $V_{i}=(v_{hl}^{i})_{1\leq h,l\leq k}$ be given by $v_{hl}^{i}=f_{hi}^{i}/\sqrt{f_{ii}^{i}}$ for $l=i$ and $v_{hl}^{i}=0$ otherwise.
By Lemma \ref{quotient of zero entropy functions}, we have $V_{i}\in M_{k}(\mathcal{E}_{0})$.

By the relations $P_{i}^*P_{i}=P_{i}$ and $P_{i}^{*}=P_{i}$, we obtain that $\sum_{h=1}^{k}|f_{hl}^i|^2=f_{ll}^i$, $1\leq l\leq k$.
Now a simple calculation leads to $V_{i}^{*}V_{i}=$diag$(0_{i-1},1,0_{k-i})$, which belongs to $\mathcal{R}_{k}(1)$, and $V_{i}V_{i}^{*}\in \mathcal{R}_{k}(1)$ as well. Since Range$(V_{i}V_i^*)\subseteq$ Range$(V_{i})\subseteq $ Range$(P_{i})$, we have $P_{i}-V_{i}V_{i}^{*}\in \mathcal{R}_{k}(r-1)$.
It is not hard to check that every element in the $i$-th row and the $i$-th column of $P_{i}-V_{i}V_{i}^{*}$ is zero.
By inductive hypothesis for $k-1$ case, there is a $U_{i}\in M_{k}(\mathcal{E}_{0})$
with every element in the $i$-th row and the $i$-th column of $U_{i}$ being zero, such that $U_{i}U_{i}^{*}=P_{i}-V_{i}V_{i}^{*}$
and $U_{i}^{*}U_{i}$ is a diagonal matrix $Q_{i}$ in $\mathcal{P}_{k}(\mathcal{E}_{0})$.

Note that $U_{i}V_{i}^*V_{i}U_{i}^*=U_{i}$diag$(0_{i-1},1,0_{k-i})U_{i}^*=0$. So $U_{i}V_{i}^*=0$. By the fact that $V_{i}^*U_{i}U_{i}^*V_{i}=V_{i}^*(P_{i}-V_{i}V_{i}^{*})V_{i}=0$,  we obtain $U_{i}^{*}V_{i}=0$.
Finally, set $U=\sum_{i=1}^{k}\chi_{C_{i}}(U_{i}+V_{i})$. Then $UU^{*}=P$ and $U^{*}U$ is a diagonal matrix $Q$ in $\mathcal{P}_{k}(\mathcal{E}_{0})$.
\end{proof}

When a diagonal matrix diag$(f_{1},\ldots,f_{k})$ in $M_{k}(\mathcal{E}_{0})$ is a projection, it satisfies that $f_{l}^2=f_{l}=\overline{f_{l}}$ for all $1\leq l\leq k$. So $f_{1},\ldots,f_{k}$ all take values in $\{0,1\}$ and they are characteristic functions.

\begin{lemma}\label{simply form}
Let $k\geq 1$. Suppose that \diag$(f_{1},\ldots,f_{k})\in \mathcal{P}_{k}(\mathcal{E}_{0})$. Then there are  characteristic functions
$g_{1},g_{2},\ldots,g_{k}\in \mathcal{E}_{0}(\mathbb{N})$ with $g_{1}\geq g_{2}\geq \cdots \geq g_{k}$, such that
\diag$(g_{1},\ldots,g_{k})\sim_{0}$ \diag$(f_{1},\ldots,f_{k})$. Moreover $\sum_{i=1}^{k}f_{i}(n)=\sum_{i=1}^{k}g_{i}(n)$ for all $n\in \mathbb{N}$.
\end{lemma}

\begin{proof}
Assume inductively that the claim holds for some $k$ with $k\geq 1$. Then we may assume that diag$(f_{1},f_{2},\ldots, f_{k+1})\sim_{0}$ diag$(f_{1},h_{2},\ldots,h_{k+1})$,
where $h_{i}=\chi_{A_{i}}$, for $i=2,\ldots,k+1$, satisfying $A_{2}\supseteq A_{3}\supseteq \cdot\cdot\cdot \supseteq A_{k+1}$
and $\sum_{i=2}^{k+1}f_{i}(n)=\sum_{i=2}^{k+1}h_{i}(n)$.
Suppose that $f_{1}=\chi_{A_{1}}$ for some $A_{1}$.
Let
$$V=\left(
     \begin{array}{cc} \vspace{.2cm}
       f_{1} & \chi_{A_{2}\setminus A_{1}}\\
       0 & \chi_{A_{2}\cap A_{1}} \\
     \end{array}
   \right)
~{\rm and}~U=\diag(V,h_{3},\ldots,h_{k}).
$$
Note that $\chi_{A_{1}\cap A_{2}}=f_{1}h_{2}$ and $\chi_{A_{2}\setminus A_{1}}=h_{2}-f_{1}h_{2}$.
So $U\in M_{k+1}(\mathcal{E}_{0})$ and
$$
U^{*}U=\diag(f_{1},h_{2},h_{3},\ldots,h_{k+1})\sim_{0}UU^{*}=\diag(\chi_{A_{1}\cup A_{2}},\chi_{A_{1}\cap A_{2}},h_{3}\ldots,h_{k+1}).
$$
It is easy to see that
$$f_{1}(n)+\sum_{i=2}^{k+1}h_{i}(n)=\chi_{A_{1}\cup A_{2}}(n)+\chi_{A_{1}\cap A_{2}}(n)+\sum_{i=3}^{k+1}h_{i}(n).$$
By induction on the $k$ case, we further assume that
$$
\diag(\chi_{A_{1}\cup A_{2}},\chi_{A_{1}\cap A_{2}},h_{3}\ldots,h_{k+1})\sim_{0}\diag(\chi_{A_{1}\cup A_{2}},g_{2}, g_{3},\ldots, g_{k+1}),
$$
where $g_{2}\geq g_{3}\geq \cdots\geq g_{k+1}$ and $\chi_{A_{1}\cap A_{2}}(n)+\sum_{i=3}^{k+1}h_{i}(n)=\sum_{i=2}^{k+1}g_{i}(n)$. Observe that $\chi_{A_{1}\cup A_{2}}(n)=0$ implies $h_{i}(n)=0$ for $i=2,\ldots,k+1$. Thus $\sum_{i=2}^{k+1}g_{i}(n)=0$, and $g_{j}(n)=0$ for all $j$ with $j=2,\ldots,k+1$.
Let $g_{1}=\chi_{A_{1}\cup A_{2}}$. Then $g_{1}\geq g_{2}\geq \cdots \geq g_{k+1}$. Now we obtain the $k+1$ case of the claim.
\end{proof}

From Lemma \ref{diagonalization} and Lemma \ref{simply form}, we see that
%
%
\[\ba{lll}
\mathcal{D}(\mathcal{E}_{0}) &=&
\{[{\rm diag}(f_{1},f_{2},\ldots,f_{k})]_{\mathcal{D}}: k\geq 1, f_{i}\in \mathcal{P}(\mathcal{E}_{0}),
i=1,\ldots,k \\
&&  {\rm and}~f_{1}\geq f_{2}\geq \cdots \geq f_{k}\}.
\ea\]
The following theorem gives a more simple description of $\mathcal{D}(\mathcal{E}_{0})$.

\begin{lemma} \label{thm:isomorphic}
The semigroup $\mathcal{D}(\mathcal{E}_{0})$ is isomorphic to the additive semigroup $\{f\in \mathcal{E}_{0}: f(\mathbb{N})\subseteq \mathbb{N}\}$.
\end{lemma}

\begin{proof}
Let Tr be the map from $\mathcal{D}(\mathcal{E}_{0})$ to $\{f\in \mathcal{E}_{0}: f(\mathbb{N})\subseteq \mathbb{N}\}$ defined by Tr$([p]_{\mathcal{D}})(n)$ $=\sum_{i=1}^{k}p_{ii}(n)$, for $p=(p_{ij})_{1\leq i,j\leq k}\in \mathcal{P}_{k}(\mathcal{E}_{0})$ and $n\in \mathbb{N}$.
We first show that Tr is well-defined, i.e., for any $p,q\in \mathcal{P}_{\infty}(\mathcal{E}_{0})$, if $[p]_{\mathcal{D}}=[q]_{\mathcal{D}}$,
then Tr$([p]_{\mathcal{D}})=$Tr$([q]_{\mathcal{D}})$. Since $p\sim_{0} q$, we have ${\rm rank}(p(n))={\rm rank}(q(n))$ for any $n$.
Thus $\sum_{i=1}^{k}p_{ii}(n)=\sum_{i=1}^{k}q_{ii}(n)$.

Now we prove that Tr is one-to-one and onto.
On one hand, suppose that $p,q\in \mathcal{P}_{\infty}(\mathcal{E}_{0})$ and Tr$([p]_{\mathcal{D}})=$Tr$([q]_{\mathcal{D}})$.
Using Lemmas \ref{diagonalization} and \ref{simply form}, we assume that for some $k\geq 1$,
$$
p\sim_{0} \diag (f_{1},f_{2},\ldots,f_{k})~ {\rm and}~q\sim_{0} \diag(g_{1},g_{2},\ldots,g_{k}),
$$
where $f_{1}\geq \cdots \geq f_{k}$ and $g_{1} \geq \cdots \geq g_{k}$.  Since Tr$([p]_{\mathcal{D}})=$Tr$([q]_{\mathcal{D}})$, we have $\sum_{i=1}^{k}f_{i}(n)=\sum_{i=1}^{k}g_{i}(n)$ for any $n\in \mathbb{N}$. Let $C_{i}=\{n:\sum_{j=1}^{k}f_{j}(n)\geq i\}$.
Then $f_{i}=\chi_{C_{i}}=g_{i}$ for $i=1,\ldots, k$. Hence $[p]_{\mathcal{D}}=[q]_{\mathcal{D}}$, i.e., the map Tr is injective.

On the other hand, for any $g\in\mathcal{E}_{0}$ with $g(\mathbb{N})\subseteq \{0,1,2,\ldots,k\}$,
let $f_{i}(n)=\chi_{\{m:g(m)\geq i\}}(n)$ for $i=1,\ldots,k$. Then $\AE(f_{i})=0$ by Proposition \ref{01 valued function} and $g=\sum_{i=1}^{k}f_{i}=$
Tr$([$diag$(f_{1},\ldots,f_{k})]_{\mathcal{D}})$. Thus Tr is a surjective map. In addition, the map Tr is a homomorphism. This is because that for any $k,l\geq 1$,
\[\ba{lll} \vspace{.2cm}
   & &\text{Tr}([\diag(f_{1},f_{2},\ldots,f_{k})]_{\mathcal{D}}+[\diag(g_{1},g_{2},\ldots,g_{l})]_{\mathcal{D}}) \\ \vspace{.2cm}
   &=& \text{Tr}([\diag(f_{1},f_2,\ldots,f_{k},g_{1},g_2,\ldots,g_{l})]_{\mathcal{D}})
   = \sum_{i=1}^{k}f_{i}+\sum_{j=1}^{l}g_{j} \\
   &=& \text{Tr}( [\diag(f_{1},f_{2},\ldots,f_{k})]_{\mathcal{D}})+\text{Tr} ([\diag(g_{1},g_{2},\ldots,g_{l})]_{\mathcal{D}}).
\ea\]
Therefore Tr is an isomorphism.
\end{proof}

The $K_0$-group $K_{0}(\cE_0)$ is the Grothendieck group of the semigroup $\mathcal{D}(\mathcal{E}_{0})$, i.e., $K_{0}(\cE_0)=\cD(\cE_0)\times \cD(\cE_0)/\sim$. We use $\<[p]_{\mathcal{D}},[q]_{\mathcal{D}}\>$ to denote the equivalence class of $([p]_{\mathcal{D}}, [q]_{\mathcal{D}})$ under $\sim$.

\begin{proof}[Proof of Theorem \ref{thm_K0_group}]
Let $\Phi$ be the map from $K_{0}(\mathcal{E}_{0})$ to the additive group $\{f\in \mathcal{E}_{0}: f(\mathbb{N})\subseteq \mathbb{Z}\}$ defined by

\[\Phi(\<[p]_{\mathcal{D}},[q]_{\mathcal{D}}\>)=\text{Tr}([p]_{\mathcal{D}})-\text{Tr}([q]_{\mathcal{D}}),~~p,q\in \mathcal{P}_{\infty}(\mathcal{E}_{0}).\]
Here the map Tr is given in Lemma \ref{thm:isomorphic}, that is Tr$([p]_{\mathcal{D}})=\sum_{i=1}^{k}p_{ii}(n)$ for any $p=(p_{ij})_{1\leq i,j\leq k}\in \mathcal{P}_{k}(\mathcal{E}_{0})$. By Lemma \ref{thm:isomorphic}, we know that Tr is an isomorphism from $\mathcal{D}(\mathcal{E}_{0})$ to the semigroup $\{f\in \mathcal{E}_{0}:f(\mathbb{N})\subseteq \mathbb{N}\}$.

We first show that $\Phi$ is well-defined. Suppose that $\<[p_{1}]_{\mathcal{D}},[p_{2}]_{\mathcal{D}}\>=\<[p_{1}']_{\mathcal{D}},[p_{2}']_{D}\>$, where $p_{1}, p_{2},p_{1}',p_{2}'\in \mathcal{P}_{\infty}(\mathcal{E}_{0})$. Then there is some $[q]_{\mathcal{D}}$ such that $[p_{1}]_{\mathcal{D}}+[p_{2}']_{\mathcal{D}}+[q]_{\mathcal{\mathcal{D}}}=[p_{1}']_{\mathcal{D}}+[p_{2}]_{\mathcal{D}}+[q]_{\mathcal{\mathcal{D}}}$. So Tr$([p_{1}]_{\mathcal{D}})$+Tr$([p_{2}']_{\mathcal{D}})=$ Tr$([p_{1}']_{\mathcal{D}})$+Tr$([p_{2}]_{\mathcal{D}})$ and $\Phi(\<[p_{1}]_{\mathcal{D}},[p_{2}]_{\mathcal{D}}\>)=\Phi(\<[p_{1}']_{\mathcal{D}},[p_{2}']_{\mathcal{D}}\>)$. Since Tr is a surjective map, it follows that $\Phi$ is surjective.

Next, we show that $\Phi$ is injective. Suppose that $p_1,p_2,p_1^\prime,p_2^\prime$ are elements in $\cP_\infty(\cE_0)$ such that $\Phi\left(\<[p_1]_\cD,[p_2]_\cD\>\right)=\Phi\left(\<[p_1^\prime]_\cD,[p_2^\prime]_\cD\>\right)$. 
Then
\[\ba{lll} \vspace{.1cm}
\text{Tr}([\diag(p_1,p_2^\prime)]_{\mathcal{D}})
&=& \text{Tr}([p_1]_{\mathcal{D}})+\text{Tr}([p_2']_{\mathcal{D}}) \\ \vspace{.1cm}
&=& \text{Tr}([p_1^\prime]_{\mathcal{D}})+\text{Tr}([p_2]_{\mathcal{D}}) \\
&=& \text{Tr}([\diag(p_1^\prime,p_2)]_{\mathcal{D}}).
\ea \]
Note that the map Tr is injective, then $[\diag(p_1,p_2^\prime)]_{\mathcal{D}}=[\diag(p_1',p_2)]_{\mathcal{D}}$. Hence $\Phi$ is injective. It is not hard to check that $\Phi$ is a group homomorphism. Thus we conclude that $\Phi$ is a group isomorphism.
\end{proof}

Next, we show $K_{1}(\mathcal{E}_{0})=0$. To prove this result, we first recall the definition of $K_{1}$-group of unital C*-algebras.
Let $\mathcal{A}$ be a unital C*-algebra with the unit $1_{\mathcal{A}}$. We use $\mathcal{U}(\mathcal{A})$ to denote the group of unitary elements of $\mathcal{A}$, i.e.,
$$
\mathcal{U}(\mathcal{A})=\{u\in \mathcal{A}: u^{*}u=uu^{*}=1_{\mathcal{A}}\}.
$$
Two elements $u,v\in \mathcal{U}(\mathcal{\mathcal{A}})$ are called homotopic in $\mathcal{U}(\mathcal{A})$,
denoted by $u\sim_{h}v$, if there is a continuous map $\varphi(t)$ from $[0,1]$ into $\mathcal{U}(\mathcal{A})$
such that $\varphi(0)=u$ and $\varphi(1)=v$. It is not hard to check that $\sim_{h}$ is an equivalence relation on $\mathcal{U}(\mathcal{A})$.
Denote by $\mathcal{U}(\mathcal{A})_{0}$ the connected component of $1_{\mathcal{A}}$ in $\mathcal{U}_{\mathcal{A}}$,
i.e., $\mathcal{U}(\mathcal{A})_{0}=\{u\in \mathcal{U}(\mathcal{A}): u\sim_{h} 1_{\mathcal{A}}\}$. It is known that $\mathcal{U}(\mathcal{A})_{0}$ is a normal subgroup of $\mathcal{U}(\mathcal{A})$.

Let $\mathcal{U}_{k}(\mathcal{A})=\mathcal{U}(M_{k}(\mathcal{A}))$. Also let $\mathcal{U}_{\infty}(\mathcal{A})=\cup_{k=1}^{\infty}\mathcal{U}_{k}(\mathcal{A})$, which is a disjoint union. For $u\in \mathcal{U}_{n}(\mathcal{A})$ and $v\in \mathcal{U}_{m}(\mathcal{A})$, we define $u\sim_{1}v$ if and only if there is a $k\geq \max\{n,m\}$ such that diag$(u,1_{k-n})\sim_{h}$ diag$(v,1_{k-m})$. It is known that $\sim_{1}$ is an equivalence relation on $\mathcal{U}_{\infty}(\mathcal{A})$. The $K_{1}$-group of $\mathcal{A}$ is defined to be $\mathcal{U}_{\infty}(\mathcal{A})/\sim_{1}$.

To prove that $K_{1}(\mathcal{E}_{0})$ is trivial, we need to show $\mathcal{U}_{k}(\mathcal{E}_{0})=\mathcal{U}_{k}(\mathcal{E}_{0})_{0}$ for all $k$. That is, for any $u\in \mathcal{U}_{k}(\mathcal{E}_{0})$, it satisfies $u\sim_{h}I_{k}$, where $I_k$ denotes the $k\times k$ diagonal matrix in $\cU_k(\cE_0)$ with diagonal elements all equal $1$. We recall some well-known results in operator algebra (see, e.g., \cite{RLL}) in the following lemma.
\begin{lemma}\label{some basic results about unitary group} Let $\mathcal{A}$ be a unital C*-algebra with the unit $1_{\mathcal{A}}$. The following propositions hold.

(\romannumeral1) Let $u\in \mathcal{U}(\mathcal{A})$. Suppose that the spectrum $sp(u)$ of $u$ is not $S^{1}$. Then $u\sim_{h}1_{\mathcal{A}}$.

(\romannumeral2) Let $u,v\in \mathcal{U}(\mathcal{A})$ and $\|u-v\|<2$, where $\|\cdot\|$ is the norm on $\mathcal{A}$. Then $u\sim_{h}v$.

(\romannumeral3) $\mathcal{U}(\mathcal{A})_{0}=\{\exp(\i v_{1})\cdots\exp(\i v_{n}): n=1,2,\ldots, v_{j}=v_{j}^{*}\in \mathcal{A},j=1,\ldots,n\}$.
\end{lemma}
In the rest of this section, we use $I_{k}$ to denote diag$(1,1,\ldots,1)$, the unit of $\mathcal{E}_{0}\otimes M_{k}(\mathbb{C})$ for $k\geq 1$.
Let $E_{ij}^{(k)}$ be the element in $\mathcal{E}_{0}\otimes M_{k}(\mathbb{C})$ obtained by exchanging the $i$-th row and the $j$-th row of $I_{k}$.
Since $sp(E_{ij}^{(k)})$ is a finite set, $E_{ij}^{(k)}\sim_{h}I_{k}$ by Lemma \ref{some basic results about unitary group}(\romannumeral1).
First, we consider the case $k=1$.
\begin{lemma}\label{the unitary group}
For any $f\in \mathcal{U}(\mathcal{E}_{0})$, we have $f\sim_{h} 1$.
\end{lemma}
\begin{proof}
Note that the spectrum of any function $\widetilde{f}\in l^{\infty}(\mathbb{N})$ is $sp(\widetilde{f})=\overline{\widetilde{f}(\mathbb{N})}$.
Suppose that $f\in \mathcal{U}(\mathcal{E}_{0})$, i.e., $|f(n)|=1$ for any $n\in \mathbb{N}$. If $\overline{f(\mathbb{N})}\neq S^{1}$, then $f\sim_{h}1$ by Lemma \ref{some basic results about unitary group}(\romannumeral1).
In the following, we assume that $\overline{f(\mathbb{N})}=S^{1}$.
Note that $\AE(f)=0$ implies $\AE(|\text{Im}(f)|)=0$.
Let $F_{1}=\{n\in \mathbb{N}: |\text{Im}(f)(n)|\leq {1}/{4}\}$ and $F_{2}=\{n\in \mathbb{N}: |\text{Im}(f)(n)|<{1}/{2}\}$. Applying Lemma \ref{subset with zero entropy},
there is a $C\subseteq \mathbb{N}$ with $F_{1}\subseteq C\subseteq F_{2}$, such that $\AE(\chi_{C})=0$.  Let $g=f\chi_{C}+1-\chi_{C}$.
We can check that $g\in \mathcal{U}(\mathcal{E}_{0})$. It is easy to see $sp(g)\neq S^{1}$. So
$g\sim_{h} 1$  by Lemma \ref{some basic results about unitary group}(\romannumeral1). Since $\|f-g\|_{l^{\infty}}=\|(1-f)(1-\chi_{C})\|_{l^{\infty}}<2$, $f\sim_{h}g$ by
Lemma \ref{some basic results about unitary group}(\romannumeral2). Thus $f\sim_{h}1$.
\end{proof}

\vspace{.4cm}

\begin{proof}[Proof of Corollary \ref{zero entropy exponential expression}]
By Lemmas \ref{some basic results about unitary group}(\romannumeral3) and \ref{the unitary group}, we have, for any $f\in \mathcal{U}(\mathcal{E}_{0})$,
there is a real-valued function $g\in \mathcal{E}_{0}$ such that $f=\exp(\i g)$, as claimed in Corollary \ref{zero entropy exponential expression}.
\end{proof}
Next we consider the case $k=2$.
\begin{lemma}\label{k=2 unitary group}
For any $u\in \mathcal{U}_{2}(\mathcal{E}_{0})$, we have $u\sim_{h} I_{2}$.
\end{lemma}
\begin{proof}
For any $u\in \cU_2(\mathcal{E}_{0})$, one may assume that
$u=\left(
\begin{matrix}
f_{1} & h\overline{f_{2}} \\
f_{2}& -h\overline{f_{1}}
\end{matrix}
\right)$ for some $f_1,f_2\in \mathcal{E}_{0}$ with $|f_{1}(n)|^2+|f_{2}(n)|^2=1$ for any $n$, and $h\in l^{\infty}(\mathbb{N})$ with $|h(n)|=1$ such that $h\overline{f_{2}},h\overline{f_{1}}\in \mathcal{E}_{0}$. Thus $h=h\overline{f_{2}}f_{2}+h\overline{f_{1}}f_{1}\in \mathcal{E}_{0}$ and $h\in \cU(\cE_0)$. By Corollary \ref{zero entropy exponential expression}, there is a real-valued function $g\in \mathcal{E}_{0}$ such that $h=\exp(\i g)$, i.e., $h(n)=\exp (\i g(n))$ for $n\in \bN$. Note that the map $\pi:\, [0,1]\rightarrow \cU_2(\cE_0)$ defined by $\pi(t)=\left(
\begin{matrix}
f_{1} & \exp(\i tg) \overline{f_{2}} \\
f_{2} & -\exp(\i tg) \overline{f_{1}}
\end{matrix}
\right)$ is continuous. We obtain that
\[
u\, \sim_{h} \, u_{1}=\left(
\begin{matrix}
f_{1} & \overline{f_{2}} \\
f_{2} & -\overline{f_{1}}
\end{matrix}
\right).
\]
The eigenvalue functions of $u_{1}$ are
\begin{equation}\label{fuctions of eigenvalues}
\lambda_{1}(n)=\i\cdot \text{Im} f_{1}(n) +\sqrt{1-(\text{Im} f_{1}(n) )^2},\quad \lambda_{2}(n)=\i\cdot \text{Im} f_{1}(n)-\sqrt{1-(\text{Im} f_{1}(n))^2}.
\end{equation}
Then $sp(u_{1})=\overline{\lambda_{1}(\mathbb{N})}\cup\overline{\lambda_{2}(\mathbb{N})}$.
Note that $\AE(|f_{1}|)=0$.
Set
\[
F_{1}=\{n\in \mathbb{N}: {1/}{2}\leq |f_{1}(n)|\leq 1\},\quad F_{2}=\{n\in \mathbb{N}: {1}/{4}<|f_{1}(n)|< {5}/{4}\}.
\]
By Lemma \ref{subset with zero entropy}, there is a subset $C$ of $\mathbb{N}$ with $F_{1}\subseteq C\subseteq F_{2}$, such that $\AE(\chi_{C})=0$.
Suppose that
\[\ba{lll}  \vspace{.1cm}
v_{1} &=& u_{1}\chi_{C}+(1-\chi_{C})\cdot \diag (1,-1)=  \left(
\begin{matrix}
g_{1} & \overline{g_{2}} \\
g_{2} & -\overline{g_{1}}
\end{matrix}
\right), \\
v_{2} &=& u_{1}(1-\chi_{C})+\chi_{C}E_{12}^{(2)}=
\left(
\begin{matrix}
h_{1} & \overline{h_{2}} \\
h_{2} & -\overline{h_{1}}
\end{matrix}
\right).
\ea\]
This construction ensures that
\[
u_1=  u_1\chi_C+u_1(1-\chi_C) =  v_1\chi_C + v_2(1-\chi_C).
\]
Note that $g_{1}(n)=f_{1}(n)\chi_{C}(n)+1-\chi_{C}(n)$, then $|g_{1}(n)|> {1}/{4}$ for all $n\in \mathbb{N}$. By Lemma \ref{quotient of zero entropy functions} and Corollary
\ref{zero entropy exponential expression}, we can rewrite $g_{1}(n)$ as $|g_{1}(n)|\exp(\i\widetilde{g_{1}}(n))$ with some real-valued function $\widetilde{g_{1}}\in \mathcal{E}_{0}$. Now we obtain
\[
v_{1}\, \sim_{h} \, \widetilde{v_{1}}=\left(
\begin{array}{cc}\vspace{.2cm}
|g_{1}(n)| & \overline{g_{2}} \\
g_{2} & -|g_{1}(n)| \\
\end{array}
\right).
\]
Note that $\text{Im}(|g_{1}(n)|)=0$. A similar argument as in \eqref{fuctions of eigenvalues} shows that $sp(\widetilde{v_{1}})$ is contained in the real line, which differs from $S^{1}$. By Lemma \ref{some basic results about unitary group}(\romannumeral1), we conclude $v_{1}\sim_{h}I_{2}$. 
Since $|f_{1}(n)|<{1}/{2}$ for $n\notin C$, one has $|h_1(n)|<1/2$ in the expression of $v_{2}$. Similar argument results in $sp(v_{2})\neq S^{1}$. So $v_{2}\sim_{h}I_{2}$. Finally, let $\varphi_i(t)$ $(i=1,2)$ be the corresponding path in $\mathcal{U}_{2}(\mathcal{E}_{0})$ connecting $v_i$ and $I_{2}$, respectively. Set $\varphi(t)=\varphi_{1}(t)\chi_{C}+\varphi_{2}(t)(1-\chi_{C})$.
Then $\varphi(t)$ is a continuous map from $[0,1]$ into $\mathcal{U}_{2}(\mathcal{E}_{0})$ with $\varphi(0)=u_{1}$ and $\varphi(1)=I_{2}$. Therefore $u\sim_{h}u_{1}\sim_{h}I_{2}$.
\end{proof}

Finally, let us prove Theorem \ref{thm_K1_group}.

\begin{proof} [Proof of Theorem \ref{thm_K1_group}]
We prove the claim by induction on $k$ that $u\sim_{h}I_{k}$ for any $u\in \mathcal{U}_{k}(\mathcal{E}_{0})$.
By Lemma \ref{the unitary group}, the statement holds for the $k=1$ case. Assume inductively that the statement has been proved for the $k-1$ case with some $k\geq 2$. Now we consider the $k$ case.
Let $u=(f_{ij})_{1\leq i,j\leq k}\in \mathcal{U}_{k}(\mathcal{E}_{0})$. Then $\sum_{i=1}^{k}|f_{i1}(n)|^2=1$ for any $n$. We apply the same method as in the proof of Lemma \ref{diagonalization}.
Set
$h_{1}=f_{11}$ and
$$F_{1}=\{n: {1}/{\sqrt{k}}\leq |h_{1}(n)|\leq 1\}, \quad G_{1}=\{n: {1}/{(2\sqrt{k})}<|h_{1}(n)|<1+{1}/{(2\sqrt{k})}\}.
$$
By Lemma \ref{subset with zero entropy}, there is a $C_{1}\subseteq \mathbb{N}$ with $F_{1}\subseteq C_{1}\subseteq G_{1}$, such that $\AE(\chi_{C_{1}})=0$. In the following, we use iteration on $l$ to construct a function $h_l$ and sets $F_l,G_l, C_l$ $(2\leq l\leq k)$.

Let $h_l=f_{l1}\cdot\prod\nolimits_{j=1}^{l-1} (1-\chi_{C_j})$. Set
\[
F_l=\{n: {1}/{\sqrt{k}}\leq |h_l(n)|\leq 1\},\quad G_l=\{n: {1}/{(2\sqrt{k})}<|h_l(n)|<1+{1}/{(2\sqrt{k})}\}.
\]
Since $\AE(h_l)=0$, there is a set $C_{l}$ with $F_l\subseteq C_l\subseteq G_l$ and $\AE(\chi_{C_l})=0$. It satisfies that $C_1,\ldots,C_k$ are pairwise disjoint and $\bN=C_1\cup \ldots \cup C_k$. Moreover, we have that $|f_{l1}(n)|<1/\sqrt{k}$ for $n\notin C_1\cup \ldots \cup C_l$ and $|f_{l1}(m)|>1/(2\sqrt{k})$ for $m\in C_l$ $(1\leq l\leq k)$.
For $i$ with $1\leq i\leq k$, let $v_{i}=u\chi_{C_{i}}+E_{i1}^{(k)}(1-\chi_{C_{i}})$.
Then $v_{i}\in \mathcal{U}_{k}(\mathcal{E}_{0})$ and the $i$-th element in the first column in $v_{i}$ is the function with
the range in $({1}/{2\sqrt{k}},1]$. In the following, we show $v_{1}\sim_{h}I_{k}$.

Suppose that
$
v_{1}=(g_{ij})_{1\leq i,j\leq k}$. Let
$$\widetilde{u_{1}}=\left(
                                      \begin{array}{rr} \vspace{.2cm}
                                        \ds\frac{\overline{g_{11}}}{\sqrt{|g_{11}|^2+|g_{21}|^2}} & \ds\frac{\overline{g_{21}}}{\sqrt{|g_{11}|^2+|g_{21}|^2}} \\
                                        \ds\frac{g_{21}}{\sqrt{|g_{11}|^2+|g_{21}|^2}}            & \ds-\frac{g_{11}}{\sqrt{|g_{11}|^2+|g_{21}|^2}}\\
                                      \end{array}
                                    \right)
$$
and $u_{1}=$ diag$(\widetilde{u_{1}},I_{k-2})$. It follows from Lemmas \ref{basic property}, \ref{quotient of zero entropy functions}, and Corollary \ref{cor_AE(f+g)_leq_AEf+AEg} that $u_{1}\in \mathcal{U}_{k}(\mathcal{E}_{0})$.
Set $u_{1}v_{1}=(g_{ij}^{(1)})_{1\leq i,j\leq k}$, then $g_{11}^{(1)}=\sqrt{|g_{11}|^2+|g_{21}|^2}$ and $g_{21}^{(1)}=0$.
Moreover, $|g_{11}^{(1)}(n)|>{1}/{(2\sqrt{k})}$ for $n\in \mathbb{N}$. So similar to the above process, we can obtain
$u_{2},\ldots,u_{k-1}\in \mathcal{U}_{k}(\mathcal{E}_{0})$, such that $u_{k-1}\cdot\cdot\cdot u_{2}u_{1}v_{1}=$ diag$(1,\widetilde{u})$,
where $\widetilde{u}\in \mathcal{U}_{k-1}(\mathcal{E}_{0})$. By Lemma \ref{k=2 unitary group}, $u_{i}\sim_{h}I_{k}$ for $i=1,\ldots, k-1$.
This implies $u_{1}^{-1}u_{2}^{-1}\cdot\cdot\cdot u_{k-1}^{-1}\sim_{h}I_{k}$.
Together with $\widetilde{u}\sim_{h}I_{k-1}$ from the inductive hypothesis on the $k-1$ case,
we obtain $v_{1}\sim_{h}I_{k}$.

For $i$ with $2\leq i\leq k$, applying the above argument to $E_{1i}^{(k)}v_{i}$,
we conclude $E_{1i}^{(k)}v_{i} \sim_{h}I_{k}$ and thus $v_{i}\sim_{h}I_{k}$. Suppose that $\varphi_{j}(t)$ is the continuous path
in $\mathcal{U}_{k}(\mathcal{E}_{0})$ from $v_{j}$ to $I_{k}$, for $j=1,\ldots,k$. Recall that $u=\sum_{j=1}^{k}u\chi_{C_{j}}=\sum_{j=1}^{k}v_{j}\chi_{C_{j}}$. Then the path $\varphi(t)=\sum_{j=1}^{k}\varphi_{j}(t)\chi_{C_{j}}$ is continuous such that $\varphi(0)=u$, $\varphi(1)=I_{k}$. The proof is completed.
\end{proof}

\bibliographystyle{amsplain}

\end{document}